\documentclass{amsart}
\usepackage{amsmath,amssymb,amsthm,bm,enumerate}
\usepackage[all]{xy}
\usepackage[dvipdfmx]{graphicx}
\usepackage[dvips]{color}
\usepackage[varg]{txfonts}

\newtheorem{theorem}{Theorem}[section]
\newtheorem{corollary}[theorem]{Corollary}

\newtheorem{lemma}[theorem]{Lemma}
\newtheorem{proposition}[theorem]{Proposition}

\theoremstyle{definition}
\newtheorem{definition}[theorem]{Definition}
\newtheorem{remark}[theorem]{Remark}

\newcommand{\R}{\mathbb{R}}
\newcommand{\Z}{\mathbb{Z}}

\newcommand{\D}{\mathcal{D}}

\newcommand{\calH}{\mathcal{H}}

\newcommand{\calS}{\mathcal{S}}

\newcommand{\calT}{\mathcal{T}}

\newcommand{\bfn}{\bm n}
\newcommand{\bfu}{\bm u}
\newcommand{\bfv}{\bm v}
\newcommand{\bfw}{\bm w}
\newcommand{\bfx}{\bm x}
\newcommand{\bfy}{\bm y}
\newcommand{\bfz}{\bm z}
\newcommand{\bfxi}{\bm \xi}

\newcommand{\abs}[1]{\left\lvert {#1} \right\rvert}

\newcommand{\Conf}{\mathrm{Conf}}

\newcommand{\emb}[2]{\mathcal{K}_{#1 ,#2}}

\newcommand{\Imm}[2]{\mathcal{I}_{#1, #2}}
\newcommand{\Int}{\mathrm{Int}\,}
\newcommand{\Inj}{\mathrm{Inj}}

\newcommand{\supp}{\mathrm{supp}}
\newcommand{\vol}{{\rm vol}}
\newcommand{\zero}{{\boldsymbol 0}}

\numberwithin{equation}{section}
\numberwithin{figure}{section}

\allowdisplaybreaks[1]

\begin{document}

\title{Lin-Wang type formula for the Haefliger invariant}
\author{Keiichi Sakai}
\email{ksakai@math.shinshu-u.ac.jp}
\address{Faculty of Science,
	Shinshu University,
	3-1-1 Asahi,
	Matsumoto, Nagano 390-8621,
	Japan}
\date{\today}
\keywords{the space of embeddings, the Haefliger invariant, configuration space integral, finite type invariant, generic immersion}

\begin{abstract}
In this paper we study the Haefliger invariant for long embeddings $\R^{4k-1}\hookrightarrow\R^{6k}$ in terms of the self-intersections of their projections to $\R^{6k-1}$, under the condition that the projection is a generic long immersion $\R^{4k-1}\looparrowright\R^{6k-1}$.
We define the notion of ``crossing changes'' of the embeddings at the self-intersections and describe the change of the isotopy classes under crossing changes using the linking numbers of the double point sets in $\R^{4k-1}$.
This formula is a higher-dimensional analogue to that of X.-S.\ Lin and Z.\ Wang for the order $2$ invariant for classical knots.
As a consequence, we show that the Haefliger invariant is of order two in a similar sense to Birman and Lin.
We also give an alternative proof for the result of M.~Murai and K.~Ohba concerning ``unknotting numbers'' of embeddings $\R^3\hookrightarrow\R^6$.
Our formula enables us to define an invariant for generic long immersions $\R^{4k-1}\looparrowright\R^{6k-1}$ which are liftable to embeddings $\R^{4k-1}\hookrightarrow\R^{6k}$.
This invariant corresponds to V.~Arnold's plane curve invariant in Lin--Wang theory, but in general our invariant does not coincide with order $1$ invariant of T.~Ekholm.
\end{abstract}


\maketitle

\section{Introduction}\label{s:intro}
A \emph{long $j$-embedding} in $\R^n$ is an embedding $\R^j\hookrightarrow\R^n$ that is the standard inclusion outside a compact set.
We denote by $\emb{n}{j}$ the space of long $j$-embeddings in $\R^n$.
Similarly, we denote the space of long immersions $\R^j\looparrowright\R^n$ by $\Imm{n}{j}$.

In \cite{K08} the author constructed, for some pairs $(n,j)$, a cochain map $I\colon\D^*\to\Omega^*_{\mathit{\mathit{DR}}}(\emb{n}{j})$ from a complex $\D^*$ of \emph{graphs} to the de Rham complex of $\emb{n}{j}$ via \emph{configuration space integrals} associated with graphs.
For other interesting pairs---in particular for $(n,j)=(6k,4k-1)$---the map $I$ has not yet been proved to be a cochain map, and it is not clear whether graph cocycles in $\D^*$ yield closed-forms of $\emb{6k}{4k-1}$.
But in \cite{K08} we found a cocycle $H\in\D^*$ and a differential form $c\in\Omega^0_{\mathit{DR}}(\emb{6k}{4k-1})$ such that $\calH :=I(H)+c\in\Omega^0_{\mathit{DR}}(\emb{6k}{4k-1})$ is closed and is equal (up to sign) to the \emph{Haefliger invariant} that gives a group isomorphism $\pi_0(\emb{6k}{4k-1})\cong\Z$.
This integral expression $\calH$ looks very similar to that for the finite type invariant $v_2$ of order $2$ (the \emph{Casson invariant}) for classical knots \cite{BottTaubes94,Kohno94}.

In this paper, based on the integral expression $\calH$, we show that the Haefliger invariant indeed behaves similarly to $v_2$.
To do this, we study $\calH(f)$ under the condition that its projection $p\circ f\in\Imm{6k-1}{4k-1}$ is a generic immersion, where $p\colon\R^{6k}\to\R^{6k-1}$ denotes the projection forgetting the last $6k$-th coordinate.
In contrast to the case of knots in $\R^3$, this is not a generic condition.
But it is possible to move the generator of $\pi_0(\emb{6k}{4k-1})$ by an isotopy so that the above condition is satisfied (see \S\ref{s:example}), and hence it is possible for all the embeddings.
A generic immersion $g\in\Imm{6k-1}{4k-1}$ has only (possibly empty) transverse two-fold self-intersection $A=A_1\sqcup\dotsb\sqcup A_m\subset\R^{6k-1}$, where each $A_i$ is a connected, closed oriented $(2k-1)$-dimensional manifold.
If $g=p\circ f$ for some $f\in\emb{6k}{4k-1}$, then $g\colon g^{-1}(A_i)\to A_i$ is a trivial double covering and we write $g^{-1}(A_i)=L_i^0(f)\sqcup L_i^1(f)\subset\R^{4k-1}$.
We define the notion of the \emph{crossing changes} at the ``crossings'' $A_i$, similarly as in the classical knot theory, and denote by $f_S\in\emb{6k}{4k-1}$ the embedding obtained from $f$ by crossing changes at $A_i$, $i\in S\subset\{1,\dotsc,m\}$.
In Theorem~\ref{thm:main1} we show that the difference $\calH(f)-\calH(f_S)$ can be described as a signed sum of the linking numbers $lk(L_i^{\epsilon}(f),L_j^{\epsilon'}(f))$ and $lk(L_i^{\epsilon}(f_S),L_j^{\epsilon'}(f_S))$, $\epsilon,\epsilon'\in\{0,1\}$, $i,j\in\{1,\dotsc,m\}$.
This formula is a higher-dimensional analogue to those for $v_2$ \cite[(4.3)]{LinWang96}, \cite[(3.2)]{Ochiai01}, \cite[(2.6)]{Ochiai04}, \cite{PolyakViro01}.

We define the notion of \emph{finite type invariants for} $\emb{6k}{4k-1}$ in a similar manner to the Birman--Lin characterization \cite{BirmanLin93}, and as a corollary of Theorem~\ref{thm:main1} we see in Theorem~\ref{thm:order2} that the Haefliger invariant is \emph{of order 2}.
In this sense, the Haefliger invariant can be seen as a higher-dimensional analogue to $v_2$.
It seems that, in some aspects, the geometric meaning of the Haefliger invariant is understood better (see, for example, \cite{Haefliger62,Budney08,Milgram72,Takase04,Takase06}) than that of finite type invariants for classical knots.
Thus more detailed studies on the Haefliger invariant (and other cohomology classes in higher dimensions that can be described by some integrals) might shed light on the geometric meaning of finite type invariants.
It might be possible that the notion of ``finite type invariants'' can also be characterized from the perspective of \emph{manifold calculus} \cite{GoodwillieWeiss99} (see also \cite{BCSS03,BCKS14,Munson05,Volic04}).
See Remark~\ref{rem:def_finite_type_inv} for some discussion.
As another consequence, we reprove the result of Murai--Ohba \cite{MuraiOhba04} concerning the ``unknotting numbers'' of long embeddings $\R^3\hookrightarrow\R^6$.

Similarly to $v_2$, the invariant $\calH$ is essentially the sum of two integrals $I(X),I(Y)$ over some configuration spaces, which correspond, respectively, to the graphs $X$ and $Y$ (see Figure~\ref{fig:H}).
Our formula in Theorem~\ref{thm:main1} is obtained by clarifying the geometric meaning of $I(X)$ to some extent;
we see in Proposition~\ref{prop:I(X)} that $I(X)$ is essentially a signed sum of the linking numbers of $L_i^{\epsilon}$'s and can be thought of as a high-dimensional analogue to the Gauss diagram term in the formulas in \cite{LinWang96,Ochiai01,Ochiai04,PolyakViro01}.
Using our formula, we can define an invariant $E$ of generic immersions $\R^{4k-1}\looparrowright\R^{6k-1}$ that can be lifted to embeddings $\R^{4k-1}\hookrightarrow\R^{6k}$ (see Theorem~\ref{thm:main2}).
In fact, $E$ is the essential part of $I(Y)$ (see \eqref{eq:E=I(Y)}).
Our invariant $E$ is a high-dimensional analogue to Lin--Wang invariant $\alpha$ for generic plane curves \cite[Definition~5.4]{LinWang96}.
The invariant $\alpha$ turns out to be a linear combination of the \emph{Arnold invariants} for generic plane curves \cite{Arnold63}.
One might therefore expect that $E$ would be a linear combination of \emph{Ekholm's order 1 invariants} \cite{Ekholm01,Ekholm01-2}, which look analogous to the Arnold invariants.
Contrary to the expectation, we see in Theorem~\ref{thm:main2} that in general $E$ is not of order $1$.
In fact, we see that, when $k=1$, the jump of $E$ at self-tangency singularities can be computed using the linking numbers of links that arise as the double point set of immersion $\R^3\looparrowright\R^5$, and the linking numbers can be arbitrarily large by the result of Ogasa \cite{Ogasa02}.

This paper is organized as follows.
In \S\ref{s:results} we fix the notation and state the results.
The main results are Theorems~\ref{thm:main1}, \ref{thm:order2} (proved in \S\ref{s:finite}), and \ref{thm:main2} (proved in \S\ref{s:gen_imm}).
In \S\ref{s:example} we show an explicit computation using Theorem~\ref{thm:main1}.
We review the graph complex $\D$ and our construction of $\calH$ in \S\ref{s:def_H}.

\section*{Acknowledgment}
The author expresses his deep appreciation to Masamichi Takase and Tadayuki Watanabe for their many fruitful suggestions.
The author also feels grateful to Dev Sinha for the discussion about the finite type invariants.
A question asked by Yuichi Yamada about $\calH$ motivated the author to start this work.
The author is partially supported by JSPS KAKENHI Grant number 25800038.

\section{Notations and results}\label{s:results}
The {\em self-intersection} $A$ of an immersion $g\colon M\looparrowright N$ is $A:=\{q\in N\mid\abs{g^{-1}(q)}\ge 2\}$, where $\abs{S}$ is the cardinality of a set $S$.
If $g\in\Imm{6k-1}{4k-1}$ is generic, then $\abs{g^{-1}(q)}=2$ for any $q\in A$.
Moreover, $A$ is a $(2k-1)$-dimensional closed submanifold and $g\colon g^{-1}(A)\to A$ is a double covering.
We call $g^{-1}(A)\subset\R^{4k-1}$ the {\em double point set}.
Suppose that $g$ is {\em liftable} to $f\in\emb{6k}{4k-1}$---namely $g=p\circ f$, where in general $p\colon\R^n\to\R^{n-1}$ is given by $p(x_1,\dotsc,x_n)=(x_1,\dotsc,x_{n-1})$---then $g\colon g^{-1}(A)\to A$ is a trivial double covering.
Let $A_i\subset\R^{6k-1}$ ($i=1,2,\dotsc$) be path components of $A$, and we call each $A_i$ a \emph{crossing} of (the ``knot diagram'' $p\circ f$ of) $f$.
We set $g^{-1}(A_i)=L_i^0\sqcup L_i^1=L_i^0(f)\sqcup L_i^1(f)$.
Each $L_i^{\epsilon}\subset\R^{4k-1}$ is a $(2k-1)$-dimensional connected closed submanifold.
By convention, $f(L_i^1)\subset\R^{6k}$ sits ``above'' $f(L_i^0)$---namely if $\bfx^{\epsilon}=(x_1,\dotsc,x_{6k-1},x_{6k}^{\epsilon})\in f(L_i^{\epsilon})$, $\epsilon=0,1$ (so $p(\bfx^0)=p(\bfx^1)$), then $x_{6k}^0<x_{6k}^1$.

\begin{remark}
Any $f\in\emb{6k}{4k-1}$ can be moved by an isotopy so that $p\circ f$ is a generic immersion; indeed such an isotopy exists for the embedding $\calS$ which generates $\pi_0(\emb{6k}{4k-1})$ (see \S\ref{s:example}).
But the condition for $f$ that $p\circ f$ is a generic immersion is not generic, and in general such an isotopy is not ``small.''
\end{remark}

\begin{remark}\label{rem:Takase}
Not all $g\in\Imm{6k-1}{4k-1}$ are regularly homotopic to any liftable immersion, in contrast to the case of plane curves.
Indeed, as shown in \cite[\S3]{Takase07}, $g\in\Imm{5}{3}$ is regularly homotopic to a liftable immersion if and only if its Smale invariant $\pi_0(\Imm{5}{3})\xrightarrow{\cong}\Z$ is even.
\end{remark}

\begin{lemma}[{\cite[Lemma~5.1.3]{Ekholm01}, \cite[Proposition~3.3]{Ekholm01-2}}]\label{lem:orientation}
For any $f\in\emb{6k}{4k-1}$ as above, the submanifolds $A_i\subset\R^{6k}$ and $L_i^{\epsilon}\subset\R^{4k-1}$ admit natural orientations.
\end{lemma}

\begin{proof}
Given a basis $\vec{u}=(\bfu_1,\dotsc,\bfu_{2k-1})$ of $T_xA_i$ ($x\in A_i$), we can choose tangent frames $\vec{v}=(\bfv_{2k},\dotsc,\bfv_{4k-1})$ and $\vec{w}=(\bfw_{2k},\dotsc,\bfw_{4k-1})$ of the two sheets of $p\circ f$ meeting at $x\in A_i$ so that $(\vec{u},\vec{v})$ and $(\vec{u},\vec{w})$ are the positive bases of these two sheets.
We say $\vec{u}$ represents the positive orientation of $A_i$ if $(\vec{u},\vec{v},\vec{w})$ is a positive basis of $\R^{6k-1}$.
Since the codimension of $p\circ f$ is even, this definition is independent of the order of the two sheets.
We orient $L_i^{\epsilon}$, $\epsilon=0,1$, so that $p\circ f\colon L_i^{\epsilon}\to A_i$ preserves the orientation.
\end{proof}

To simplify the computations, we often move $f\in\emb{6k}{4k-1}$ to a special position.

\begin{definition}[see Figure~\ref{fig:special_position}]\label{def:special_position}
We say an embedding $f\in\emb{6k}{4k-1}$ is {\em almost planar} if
\begin{enumerate}[(i)]
\item
 the composite $p\circ f\colon \R^{4k-1}\looparrowright\R^{6k-1}$ is a generic immersion,
\item
 $f(\R^{4k-1})\subset\R^{6k-1}\times[0,\delta]$ for a small $\delta>0$, and
\item
 $f(\R^{4k-1}\setminus\bigcup_iN(L_i^1))\subset\R^{6k-1}\times\{0\}$, where $N(L_i^{\epsilon})\subset\R^{4k-1}$ are closed tubular neighborhoods of $L_i^{\epsilon}$ in $\R^{4k-1}$ such that $N(L_i^{\epsilon})\cap N(L_j^{\epsilon'})=\emptyset$ if $(i,\epsilon)\ne(j,\epsilon')$.
\end{enumerate}
\end{definition}
\begin{figure}[htb]
\centering
\unitlength 0.1in
\begin{picture}( 36.5000,  8.0000)(  6.0000,-14.0000)
%
{\color[named]{Black}{%
\special{pn 8}%
\special{pa 4200 1400}%
\special{pa 4200 1100}%
\special{fp}%
\special{sh 1}%
\special{pa 4200 1100}%
\special{pa 4180 1168}%
\special{pa 4200 1154}%
\special{pa 4220 1168}%
\special{pa 4200 1100}%
\special{fp}%
}}%
\put(42.5000,-11.0000){\makebox(0,0)[lt]{$x_{6k}$}}%
%
{\color[named]{Black}{%
\special{pn 8}%
\special{pa 600 800}%
\special{pa 2080 800}%
\special{fp}%
}}%
%
{\color[named]{Black}{%
\special{pn 8}%
\special{pa 600 800}%
\special{pa 1200 1400}%
\special{fp}%
}}%
%
{\color[named]{Black}{%
\special{pn 8}%
\special{pa 3600 800}%
\special{pa 4200 1400}%
\special{fp}%
}}%
%
{\color[named]{Black}{%
\special{pn 8}%
\special{pa 1200 1400}%
\special{pa 4200 1400}%
\special{fp}%
}}%
%
{\color[named]{Black}{%
\special{pn 13}%
\special{pa 2800 1300}%
\special{pa 2040 920}%
\special{fp}%
}}%
%
{\color[named]{Black}{%
\special{pn 13}%
\special{pa 1960 880}%
\special{pa 1860 820}%
\special{fp}%
}}%
%
{\color[named]{Black}{%
\special{pn 8}%
\special{pa 2200 800}%
\special{pa 2600 800}%
\special{fp}%
}}%
%
{\color[named]{Black}{%
\special{pn 8}%
\special{pa 2750 800}%
\special{pa 3600 800}%
\special{fp}%
}}%
%
{\color[named]{Black}{%
\special{pn 8}%
\special{pa 2400 700}%
\special{pa 2400 1100}%
\special{dt 0.045}%
}}%
\put(24.0000,-7.0000){\makebox(0,0){$\times$}}%
\put(28.7000,-6.3000){\makebox(0,0)[lt]{$f(L_i^1)$}}%
\put(22.0000,-12.3000){\makebox(0,0)[rt]{$A_i=f(L_i^0)$}}%
\put(36.5000,-13.0000){\makebox(0,0){$\R^{6k-1}$}}%
\put(16.5000,-10.8000){\makebox(0,0)[rb]{$f(\R^{4k-1})$}}%
%
{\color[named]{Black}{%
\special{pn 4}%
\special{pa 2200 1250}%
\special{pa 2400 1100}%
\special{fp}%
}}%
\put(24.0000,-11.0000){\makebox(0,0){$\bullet$}}%
{\color[named]{Black}{%
\special{pn 13}%
\special{pa 1600 1100}%
\special{pa 1630 1100}%
\special{pa 1636 1098}%
\special{pa 1640 1098}%
\special{pa 1660 1094}%
\special{pa 1666 1094}%
\special{pa 1670 1092}%
\special{pa 1690 1088}%
\special{pa 1696 1086}%
\special{pa 1700 1086}%
\special{pa 1706 1084}%
\special{pa 1710 1082}%
\special{pa 1730 1074}%
\special{pa 1736 1074}%
\special{pa 1740 1072}%
\special{pa 1746 1068}%
\special{pa 1760 1062}%
\special{pa 1766 1060}%
\special{pa 1776 1056}%
\special{pa 1780 1052}%
\special{pa 1786 1050}%
\special{pa 1800 1042}%
\special{pa 1806 1040}%
\special{pa 1840 1018}%
\special{pa 1846 1014}%
\special{pa 1856 1008}%
\special{pa 1860 1004}%
\special{pa 1870 998}%
\special{pa 1876 994}%
\special{pa 1880 992}%
\special{pa 1886 988}%
\special{pa 1890 984}%
\special{pa 1896 980}%
\special{pa 1900 978}%
\special{pa 1910 970}%
\special{pa 1916 966}%
\special{pa 1936 950}%
\special{pa 1940 948}%
\special{pa 1966 928}%
\special{pa 1970 924}%
\special{pa 2030 876}%
\special{pa 2036 874}%
\special{pa 2060 854}%
\special{pa 2066 850}%
\special{pa 2086 834}%
\special{pa 2090 832}%
\special{pa 2100 824}%
\special{pa 2106 820}%
\special{pa 2110 816}%
\special{pa 2116 814}%
\special{pa 2120 810}%
\special{pa 2126 806}%
\special{pa 2130 802}%
\special{pa 2140 796}%
\special{pa 2146 792}%
\special{pa 2156 786}%
\special{pa 2160 782}%
\special{pa 2196 762}%
\special{pa 2200 760}%
\special{pa 2216 750}%
\special{pa 2220 748}%
\special{pa 2226 746}%
\special{pa 2236 742}%
\special{pa 2240 738}%
\special{pa 2256 732}%
\special{pa 2260 730}%
\special{pa 2266 728}%
\special{pa 2270 726}%
\special{pa 2290 718}%
\special{pa 2296 718}%
\special{pa 2300 716}%
\special{pa 2306 714}%
\special{pa 2310 712}%
\special{pa 2330 708}%
\special{pa 2336 706}%
\special{pa 2340 706}%
\special{pa 2360 702}%
\special{pa 2380 702}%
\special{pa 2386 700}%
\special{pa 2400 700}%
\special{fp}%
}}%
{\color[named]{Black}{%
\special{pn 13}%
\special{pa 2400 700}%
\special{pa 2416 700}%
\special{pa 2420 702}%
\special{pa 2440 702}%
\special{pa 2460 706}%
\special{pa 2466 706}%
\special{pa 2470 708}%
\special{pa 2490 712}%
\special{pa 2496 714}%
\special{pa 2500 716}%
\special{pa 2506 718}%
\special{pa 2510 718}%
\special{pa 2530 726}%
\special{pa 2536 728}%
\special{pa 2540 730}%
\special{pa 2546 732}%
\special{pa 2560 738}%
\special{pa 2566 742}%
\special{pa 2576 746}%
\special{pa 2580 748}%
\special{pa 2586 750}%
\special{pa 2600 760}%
\special{pa 2606 762}%
\special{pa 2640 782}%
\special{pa 2646 786}%
\special{pa 2656 792}%
\special{pa 2660 796}%
\special{pa 2670 802}%
\special{pa 2676 806}%
\special{pa 2680 810}%
\special{pa 2686 814}%
\special{pa 2690 816}%
\special{pa 2696 820}%
\special{pa 2700 824}%
\special{pa 2710 832}%
\special{pa 2716 834}%
\special{pa 2736 850}%
\special{pa 2740 854}%
\special{pa 2766 874}%
\special{pa 2770 876}%
\special{pa 2830 924}%
\special{pa 2836 928}%
\special{pa 2860 948}%
\special{pa 2866 950}%
\special{pa 2886 966}%
\special{pa 2890 970}%
\special{pa 2900 978}%
\special{pa 2906 980}%
\special{pa 2910 984}%
\special{pa 2916 988}%
\special{pa 2920 992}%
\special{pa 2926 994}%
\special{pa 2930 998}%
\special{pa 2940 1004}%
\special{pa 2946 1008}%
\special{pa 2956 1014}%
\special{pa 2960 1018}%
\special{pa 2996 1040}%
\special{pa 3020 1052}%
\special{pa 3026 1056}%
\special{pa 3036 1060}%
\special{pa 3040 1062}%
\special{pa 3056 1068}%
\special{pa 3060 1072}%
\special{pa 3066 1074}%
\special{pa 3070 1074}%
\special{pa 3090 1082}%
\special{pa 3096 1084}%
\special{pa 3100 1086}%
\special{pa 3106 1086}%
\special{pa 3110 1088}%
\special{pa 3130 1092}%
\special{pa 3136 1094}%
\special{pa 3140 1094}%
\special{pa 3160 1098}%
\special{pa 3166 1098}%
\special{pa 3170 1100}%
\special{pa 3200 1100}%
\special{fp}%
}}%
%
{\color[named]{Black}{%
\special{pn 13}%
\special{pa 1400 1100}%
\special{pa 1610 1100}%
\special{fp}%
}}%
%
{\color[named]{Black}{%
\special{pn 13}%
\special{pa 3190 1100}%
\special{pa 3400 1100}%
\special{fp}%
}}%
%
{\color[named]{Black}{%
\special{pn 4}%
\special{ar 2650 750 250 150  3.7305178 5.6942602}%
}}%
\end{picture}%
\caption{An almost planar embedding}
\label{fig:special_position}
\end{figure}
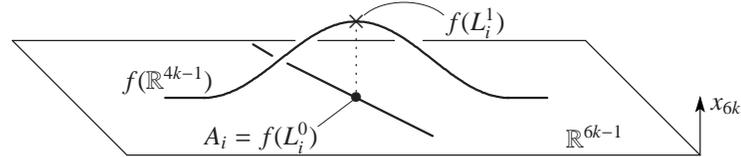
If $f\in\emb{6k}{4k-1}$ is such that $p\circ f\in\Imm{6k-1}{4k-1}$ is generic, then we can transform $f$ to be almost planar without changing $p\circ f$, by an isotopy in the $x_{6k}$-direction.
Notice that if $f$ is almost planar, then the crossings $A_i$ are equal to $f(L_i^0)$.

Suppose that $f$ is almost planar and that the self-intersection of $p\circ f$ has $m$ components.
For $S\subset\{1,\dotsc,m\}$, let $f_S\in\emb{6k}{4k-1}$ be defined by
\[
 f_S(x):=
 \begin{cases}
  \iota(f(x)) & x\in N(L_i^1),\ i\in S,\\
  f(x)        & \text{otherwise},
 \end{cases}
\]
where $\iota\colon \R^{6k}\to\R^{6k}$ is given by $\iota(x_1,\dotsc,x_{6k})=(x_1,\dotsc,x_{6k-1},-x_{6k})$.
We say $f_S$ is an embedding obtained from $f$ by {\em crossing changes} at the crossings $\{A_i\}_{i\in S}$.
Notice that $p\circ f=p\circ f_S$
and
\[
 L_i^{\epsilon}(f_S)=
 \begin{cases}
  L_i^{\epsilon+1}(f)	& i\in S, \\
  L_i^{\epsilon}(f)	& i\not\in S;
 \end{cases}
\]
here $\epsilon$ is understood to be in $\Z/2=\{0,1\}$ and $1+1=0$.

Let $\calH\colon \emb{6k}{4k-1}\to\Z$ be the Haefliger invariant (see \S\ref{s:def_H} for our construction).
Our main theorem describes the difference $\calH(f)-\calH(f_S)$ using the linking numbers of $L_i^{\epsilon}$'s.

\begin{theorem}\label{thm:main1}
Let $f\in\emb{6k}{4k-1}$ be such that $p\circ f$ is a generic immersion and has the nonempty self-intersection $A=A_1\sqcup\dotsb\sqcup A_m$.
Then for any $S\subset\{1,\dotsc,m\}$,
\begin{align}
 &\calH(f)-\calH(f_S)\notag\\
 &\ =\frac{1}{4}\Bigl(\sum_{(i,\epsilon)<(j,\epsilon')}(-1)^{\epsilon+\epsilon'}lk(L^{\epsilon}_i(f),L^{\epsilon'}_j(f))
 -\sum_{(i,\epsilon)<(j,\epsilon')}(-1)^{\epsilon+\epsilon'}lk(L^{\epsilon}_i(f_S),L^{\epsilon'}_j(f_S))\Bigr)\label{eq:main1}\\
 &\ =\frac{1}{2}\sum_{\genfrac{}{}{0pt}{1}{(i,\epsilon)<(j,\epsilon'),}{(\text{exactly one of }i,j)\in S}}(-1)^{\epsilon+\epsilon'}lk(L^{\epsilon}_i(f),L^{\epsilon'}_j(f)),\label{eq:main2}
\end{align}
where $lk$ stands for the linking number, and we write $(i,\epsilon)<(j,\epsilon')$ if $i<j$ or if $i=j$, $\epsilon=0$, $\epsilon'=1$.
\end{theorem}

\eqref{eq:main2} follows from \eqref{eq:main1};
if $i,j\not\in S$, then $lk(L_i^{\epsilon}(f),L_j^{\epsilon'}(f))=lk(L_i^{\epsilon}(f_S),L_j^{\epsilon'}(f_S))$ is contained in both sums in \eqref{eq:main1} with the same sign $(-1)^{\epsilon+\epsilon'}$ and cancels out.
If $i,j\in S$, then $lk(L_i^{\epsilon}(f),L_j^{\epsilon'}(f))=lk(L_i^{\epsilon+1}(f_S),L_j^{\epsilon'+1}(f_S))$ is contained in both sums with the same sign $(-1)^{\epsilon+\epsilon'}$ and cancels out.
If $i\in S$ and $j\not\in S$, then $lk(L_i^{\epsilon}(f),L_j^{\epsilon'}(f))$ is contained in the first sum with sign $(-1)^{\epsilon+\epsilon'}$ while $lk(L_i^{\epsilon}(f),L_j^{\epsilon'}(f))=lk(L_i^{\epsilon+1}(f_S),L_j^{\epsilon'}(f_S))$ is contained in the second sum with the opposite sign $(-1)^{\epsilon+1+\epsilon'}$.

\begin{remark}\label{rem:LinWang}
The formula in Theorem~\ref{thm:main1} is similar to Lin and Wang's formula for the Casson invariant $v_2$ \cite{LinWang96};
given a diagram of $f\in\emb{3}{1}$, let $f_S\in\emb{3}{1}$ be obtained by changing the crossings $c_i$, $i\in S$.
Then a slight generalization of \cite[(4.3)]{LinWang96} can be written as
\begin{equation}\label{eq:LinWang_original}
 v_2(f)-v_2(f_S)
	=\frac{1}{4}\Bigl\langle\raisebox{-0.2\height}{\includegraphics[scale=0.7]{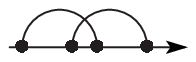}},G(f)\Bigr\rangle
	-\frac{1}{4}\Bigl\langle\raisebox{-0.2\height}{\includegraphics[scale=0.7]{X_graph_unsigned.eps}},G(f_S)\Bigr\rangle,
\end{equation}
where $G(f)$ is the \emph{Gauss diagram} of (the diagram of) $f$ and
$\Bigl\langle\raisebox{-0.2\height}{\includegraphics[scale=0.7]{X_graph_unsigned.eps}},G(f)\Bigr\rangle$
is the sum of $\varepsilon_1\varepsilon_2$ for all the subdiagrams of $G(f)$ of the shape \raisebox{-0.2\height}{\includegraphics{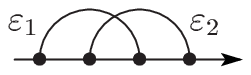}} ($\varepsilon_i=\pm 1$ are the signs of the corresponding crossings of the diagrams).
This kind of pairing appears elsewhere, for example, in \cite{PolyakViro01}.
Choosing $S$ that yields a ``descending diagram'' $f_S$ (and hence $f_S$ is trivial), and computing the right-hand side of \eqref{eq:LinWang_original}, we reprove the Polyak--Viro formula \cite[Theorem~1.A]{PolyakViro01}.
Regarding the pairing as the sum of the linking numbers of ``$0$-dimensional Hopf links $S^0\sqcup S^0\hookrightarrow\R^1$'' determined by the crossings of the diagrams, we can say that Theorem~\ref{thm:main1} is a higher-dimensional analogue to \eqref{eq:LinWang_original} and to the Polyak--Viro formula.
\end{remark}

Theorem~\ref{thm:main1} together with the result of Ogasa \cite{Ogasa02} gives an alternative proof for the following result of Murai--Ohba \cite{MuraiOhba04}, which states that the ``unknotting number'' of any nontrivial embedding $f\in\emb{6}{3}$ is $1$ (see \S\ref{ss:k=1} for the proof).

\begin{corollary}[{\cite{MuraiOhba04}}]\label{cor:Ohba}
Any nontrivial $f\in\emb{6}{3}$ can be unknotted by a crossing change at a single crossing.
Namely, $f$ is isotopic to some $f'$, with $f'_{\{1\}}$ isotopic to the trivial inclusion.
\end{corollary}

The proof is outlined as follows:
Given any two-component link $L\subset\R^3$, by a result of Ogasa \cite{Ogasa02} we can find $f_0\colon \R^3\hookrightarrow\R^6$ such that it is isotopic to the trivial inclusion and its projection $p\circ f\colon \R^3\looparrowright\R^5$ has $L$ as its double point set.
By Theorem~\ref{thm:main1} the embedding $f_1$ obtained by the crossing change along $L$ satisfies $\calH(f_1)=-lk(L)$.
This means that any $f\in\emb{6}{3}$ with arbitrary $\calH(f)$ can be obtained by a single crossing change from $f_0$.
See \S\ref{ss:k=1} for details.
In \cite{MuraiOhba04} an explicit way to unknot the generator of $\pi_0(\emb{6}{3})$ (and its connected sums) by a single crossing change is given.

Below we introduce the notion of {\em finite type invariants} for $\emb{6k}{4k-1}$.
As a consequence of Theorem~\ref{thm:main1}, we prove in \S\ref{s:finite} that $\calH$ is an invariant of order $2$.

\begin{definition}\label{def:finite_type_invariant}
Let $u\colon\emb{6k}{4k-1}\to\R$ be a function, and let $A=\{A_i\}_{1\le i\le s+1}$ be a (sub)set of crossings of $p\circ f$, where $f\in\emb{6k}{4k-1}$ is almost planar.
Define
\[
 V_{s+1}(u)(f):=\sum_{S\subset\{1,\dotsc,s+1\}}(-1)^{\abs{S}}u(f_S).
\]
An isotopy invariant $u$ is said to be {\em of order $s$} if $V_{s+1}(u)=0$.
\end{definition}

\begin{theorem}\label{thm:order2}
The Haefliger invariant is of order 2.
\end{theorem}

\begin{remark}\label{rem:def_finite_type_inv}
Our Definition~\ref{def:finite_type_invariant} of ``finite type invariants'' is modeled after \cite{BirmanLin93} and is similar to those in \cite{HabiroKanenobuShima98,Watanabe07}.
Finite type invariants can also in some cases be characterized as the functions that factor through stages of the \emph{Taylor tower} \cite{GoodwillieWeiss99}.
Voli\'c uses Bott--Taubes integrals to prove in \cite[Theorem~4.5]{Volic06} that $\R$-valued invariants of order $k$ factor through the $2k$-th stage of the homology Taylor tower for $\emb{3}{1}$.
In \cite{BCSS03} it is proved that $v_2\colon \emb{3}{1}\to\Z$ factors through the third stage of the homotopy Taylor tower for $\emb{3}{1}$.
Budney, Conant, Koytcheff, and Sinha \cite{BCKS14} show that the ($k+1$)-st stage of the homotopy tower defines order $k$ invariants, and based on spectral sequence calculations conjecture any additive invariant of order $k$ factors through this tower.

In general, the equivalence between these two characterizations of finiteness, Birman--Lin and Goodwillie--Weiss, is not known.
Theorem~\ref{thm:order2} together with a result of Munson \cite{Munson05} proves the equivalence for $\emb{6k}{4k-1}$.
\end{remark}

In \S\ref{s:gen_imm} we define an invariant for generic \emph{liftable} immersion $\R^{4k-1}\looparrowright\R^{6k-1}$ and discuss some of its properties.

\begin{theorem}\label{thm:main2}
Let $g\in\Imm{6k-1}{4k-1}$ be a liftable generic immersion---namely, $g=p\circ f$ for some $f\in\emb{6k}{4k-1}$.
Choose a lift $f\in\emb{6k}{4k-1}$ and define
\begin{equation}\label{eq:E=H-link}
 E(g):=\calH(f)-\frac{1}{4}\sum_{(i,\epsilon)<(j,\epsilon')}(-1)^{\epsilon+\epsilon'}lk(L^{\epsilon}_i,L_j^{\epsilon'}).
\end{equation}
Then $E$ is independent of the choice of $f$ and is invariant of generic immersions.
$E$ varies at the strata of non-generic immersions as described in Lemmas~\ref{lem:jump_definite}, \ref{lem:jump_indefinite}, and \ref{lem:jump_triple}.
In the case $k=1$, the invariant $E$ is not of order 1 in the sense of Ekholm \cite{Ekholm01,Ekholm01-2}.
\end{theorem}

$E$ is an invariant of generic immersions because the linking numbers are constant unless the isotopy class of the self-intersection of $p\circ f$ changes.
That $E$ is well defined follows from our formula in Theorem~\ref{thm:main1}.
See \S\ref{ss:E_invariant}.

\begin{remark}
The invariant $E$ is a high-dimensional analogue to the Lin--Wang invariant $\alpha$ \cite[Definition~5.4]{LinWang96} for generic plane curves;
let $g\in\Imm{2}{1}$ be a generic plane curve, and let $f\in\emb{3}{1}$ be its lift---namely, $g=p\circ f$.
Then \eqref{eq:LinWang_original} implies that
\[
 \alpha(g):=v_2(f)-\frac{1}{4}\Bigl\langle\raisebox{-0.2\height}{\includegraphics[scale=0.7]{X_graph_unsigned.eps}},G(f)\Bigr\rangle
\]
is independent of the choice $f$.
The invariant $\alpha$ is in fact equal to a linear combination of the Arnold invariants $J^{\pm}$ and $\mathit{St}$ of a plane curve \cite{Arnold63}.
On the other hand Ekholm \cite[\S6.1]{Ekholm01}, \cite[\S4]{Ekholm01-2} defined invariants for generic immersions $M^{nm-1}\looparrowright N^{n(m+1)-1}$ that behave similarly to the Arnold invariants (for us, $m=2$, $n=2k$).
One may thus expect that $E$ might be a linear combination of Ekholm invariants, but Theorem~\ref{thm:main2} says that in general it is not the case.
\end{remark}
\section{Example}\label{s:example}
Using Theorem~\ref{thm:main1}, we show $\calH(\calS)=\pm1$ for Haefliger's generator $\calS$ of $\pi_0(\emb{6k}{4k-1})\cong\Z$ \cite{Haefliger62}.
Fix $\alpha,\beta>0$ so that $2\beta<\alpha$.
Consider the {\em Borromean ring} $X\sqcup Y\sqcup Z\subset\R^{6k}$, where
\begin{align*}
 X&:=\partial\{(\zero,\bfy,\bfz)\in(\R^{2k})^{\times 3}\mid\abs{\bfy}\le\alpha,\,\abs{\bfz}\le\beta\}\approx S^{4k-1},\\
 Y&:=\partial\{(\bfx,\zero,\bfz)\in(\R^{2k})^{\times 3}\mid\abs{\bfz}\le\alpha,\,\abs{\bfx}\le\beta\}\approx S^{4k-1},\\
 Z&:=\partial\{(\bfx,\bfy,\zero)\in(\R^{2k})^{\times 3}\mid\abs{\bfx}\le\alpha,\,\abs{\bfy}\le\beta\}\approx S^{4k-1}
\end{align*}
(see Figure~\ref{fig:Borromean}), and smooth their corners to get smooth $(4k-1)$-spheres (denoted by $X,Y,Z$ again).
\begin{figure}[htb]
\centering
\input{Borromean}
\caption{Haefliger's generator $\calS$ and the self-intersection of $p\circ\calS$}
\label{fig:Borromean}
\end{figure}
$\calS$ is defined as the connected-sum $\calS:=X\sharp Y\sharp Z\sharp f_0$, where $f_0\colon\R^{4k-1}\subset\R^{6k}$ is (isotopic to) the standard inclusion.

Let $\bfn:=(1,\dotsc,1)\in\R^{6k}$, and consider the projection $p\colon\R^{6k}\to(\R\bfn)^{\perp}$, instead of $\R^{6k}\to\R^{6k-1}\times\{0\}$.
Then $p\circ\calS$ is generic, as seen in Figure~\ref{fig:Borromean}.
To detect $p(X)\cap p(Y)$, find $(\zero,\bfy,\bfz)\in X$ and $t\in\R$ satisfying $(\zero,\bfy,\bfz)+t\bfn\in Y$.
In fact, $p(X)\cap p(Y)=A_1\sqcup A_2$ has two components, and the double point set $L_i^{\epsilon}\subset\R^{4k-1}$ satisfying $A_i=p(\calS(L_i^0))=p(\calS(L_i^1))$ ($i=1,2$) is given as follows:
put $\bfn':=(1,\dotsc,1)\in\R^{2k}$ and $\beta':=\beta/\sqrt{2k}$; then
\begin{alignat*}{3}
 &L_1^0=\{(-\beta'\bfn',\zero,\bfz)\in Y\mid\abs{\bfz+\beta'\bfn'}=\beta\},& \quad
 &L_1^1=\{(\zero,\beta'\bfn',\bfz)\in X\mid\abs{\bfz}=\beta\},& \\
 &L_2^0=\{(\zero,-\beta'\bfn',\bfz)\in X\mid\abs{\bfz}=\beta\},& \quad
 &L_2^1=\{(\beta'\bfn',\zero,\bfz)\in Y\mid\abs{\bfz-\beta'\bfn'}=\beta\}& 
\end{alignat*}
(in the computation we use $2\beta<\alpha$).
By symmetry we see that $p(Y)\cap p(Z)=A_3\sqcup A_4$ has two components, and the double point sets satisfying $A_i=p(\calS(L_i^0))=p(\calS(L_i^1))$ ($i=3,4$) are given as
\begin{alignat*}{3}
 &L_3^0=\{(\bfx,-\beta'\bfn',\zero)\in Z\mid\abs{\bfx+\beta'\bfn'}=\beta\},& \quad
 &L_3^1=\{(\bfx,\zero,\beta'\bfn')\in Y\mid\abs{\bfx}=\beta\},& \\
 &L_4^0=\{(\bfx,\zero,-\beta'\bfn')\in Y\mid\abs{\bfx}=\beta\},& \quad
 &L_4^1=\{(\bfx,\beta'\bfn',\zero)\in Z\mid\abs{\bfx-\beta'\bfn'}=\beta\}.&
\end{alignat*}
Similarly, $p(Z)\cap p(X)=A_5\sqcup A_6$ satisfies $A_i=p(\calS(L_i^0))=p(\calS(L_i^1))$ ($i=5,6$), where
\begin{alignat*}{3}
 &L_5^0=\{(\zero,\bfy,-\beta'\bfn')\in X\mid\abs{\bfy+\beta'\bfn'}=\beta\},& \quad
 &L_5^1=\{(\beta'\bfn',\bfy,\zero)\in Z\mid\abs{\bfy}=\beta\},&\\
 &L_6^0=\{(-\beta'\bfn',\bfy,\zero)\in Z\mid\abs{\bfy}=\beta\},& \quad
 &L_6^1=\{(\zero,\bfy,\beta'\bfn')\in X\mid\abs{\bfy-\beta'\bfn'}=\beta\}.&
\end{alignat*}
If we take the connected-sum in the construction of $\calS$ suitably, then there are no self-intersection of $p\circ\calS$ other than $A_1\sqcup\dotsb\sqcup A_6$.
The corner smoothing does not cause any trouble; for example, any $(\zero,\bfy,-\beta'\bfn')\in L^0_5$ satisfies $\abs{\bfy}\le 2\beta<\alpha$ and $L^0_5$ does not touch the corner.
All $L_i^{\epsilon}$ are $S^{2k-1}$ and they form six disjoint Hopf links
\[
 L_1^1\sqcup L_6^1,\ L_2^0\sqcup L_5^0\subset X,\quad
 L_1^0\sqcup L_4^0,\ L_2^1\sqcup L_3^1\subset Y,\quad
 L_3^0\sqcup L_6^0,\ L_4^1\sqcup L_5^1\subset Z,
\]
all of whose linking numbers are by symmetry equal to each other.
Since we can unknot $\calS$ by the crossing change at $A_1$, by \eqref{eq:main2}
\[
 \calH(\calS)=\frac{1}{2}((-1)^{1+1}lk(L^1_1,L^1_6)+(-1)^{0+0}lk(L^0_1,L^0_4))=\pm\frac{1}{2}(1+1)=\pm1.
\]
\section{A review of an integral expression of the Haefliger invariant}\label{s:def_H}
\subsection{Graph complex and configuration space integral}
Here we briefly recall the cochain complex $\D^*=\D^*_{n,j}$ of \emph{oriented graphs} and the linear map $I\colon\D^*\to\Omega^*_{\mathit{DR}}(\emb{n}{j})$ that the author defined in \cite{K08} generalizing those in \cite{CCL02,CattaneoRossi05,Watanabe07}.
The map $I$ is a cochain map under some conditions on $n,j$ and graphs.
See \cite{K08, KWatanabe12} for details.

By a \emph{graph} we mean a graph with two types of vertices (called \emph{i-} and \emph{e-vertices}) and two types of edges (called $\eta$- and $\theta$\emph{-edges}).
The sets of i- and e-vertices, $\eta$- and $\theta$-edges of a graph $\Gamma$ are denoted by, respectively, $V_{\rm i}(\Gamma)$, $V_{\rm e}(\Gamma)$, $E_{\eta}(\Gamma)$, and $E_{\theta}(\Gamma)$.
We give the weights $j$, $n$, $j-1$, and $n-1$ to the elements of $V_{\rm i}(\Gamma)$, $V_{\rm e}(\Gamma)$, $E_{\eta}(\Gamma)$, and $E_{\theta}(\Gamma)$.
An \emph{orientation} of a graph is a choice of ordering of the weighted set $V_{\rm i}(\Gamma)\sqcup V_{\rm e}(\Gamma)\sqcup E_{\eta}(\Gamma)\sqcup E_{\theta}(\Gamma)$ together with the orientations of edges, modulo even permutations.
Reversing $\eta$- and $\theta$-edges changes the sign of orientation of graphs by $(-1)^j$ and $(-1)^n$, respectively.
Figure~\ref{fig:H} shows examples of graphs;
i- and e-vertices are depicted by $\bullet$ and $\circ$, respectively, and $\eta$- and $\theta$-edges are depicted by solid and dotted arrows, respectively (Figure~\ref{fig:H} shows graphs for even $n$ and the orientation of the $\theta$-edges are omitted).
\begin{figure}[htb]
\centering
\unitlength 0.1in
\begin{picture}( 26.9000,  7.6000)(  5.4000,-14.3000)
%
{\color[named]{Black}{%
\special{pn 0}%
\special{sh 1.000}%
\special{ia 800 800 30 30  0.0000000 6.2831853}%
}}%
{\color[named]{Black}{%
\special{pn 8}%
\special{ar 800 800 30 30  0.0000000 6.2831853}%
}}%
%
{\color[named]{Black}{%
\special{pn 8}%
\special{pa 830 1400}%
\special{pa 1370 1400}%
\special{fp}%
\special{sh 1}%
\special{pa 1370 1400}%
\special{pa 1304 1380}%
\special{pa 1318 1400}%
\special{pa 1304 1420}%
\special{pa 1370 1400}%
\special{fp}%
}}%
%
{\color[named]{Black}{%
\special{pn 8}%
\special{pa 800 830}%
\special{pa 800 1370}%
\special{dt 0.045}%
}}%
%
{\color[named]{Black}{%
\special{pn 8}%
\special{pa 1400 830}%
\special{pa 1400 1370}%
\special{dt 0.045}%
}}%
\put(8.5000,-11.0000){\makebox(0,0){$1$}}%
\put(13.5000,-11.0000){\makebox(0,0){$2$}}%
%
{\color[named]{Black}{%
\special{pn 0}%
\special{sh 1.000}%
\special{ia 800 1400 30 30  0.0000000 6.2831853}%
}}%
{\color[named]{Black}{%
\special{pn 8}%
\special{ar 800 1400 30 30  0.0000000 6.2831853}%
}}%
%
{\color[named]{Black}{%
\special{pn 0}%
\special{sh 1.000}%
\special{ia 1400 1400 30 30  0.0000000 6.2831853}%
}}%
{\color[named]{Black}{%
\special{pn 8}%
\special{ar 1400 1400 30 30  0.0000000 6.2831853}%
}}%
%
{\color[named]{Black}{%
\special{pn 0}%
\special{sh 1.000}%
\special{ia 1400 800 30 30  0.0000000 6.2831853}%
}}%
{\color[named]{Black}{%
\special{pn 8}%
\special{ar 1400 800 30 30  0.0000000 6.2831853}%
}}%
%
{\color[named]{Black}{%
\special{pn 0}%
\special{sh 1.000}%
\special{ia 2400 800 30 30  0.0000000 6.2831853}%
}}%
{\color[named]{Black}{%
\special{pn 8}%
\special{ar 2400 800 30 30  0.0000000 6.2831853}%
}}%
%
{\color[named]{Black}{%
\special{pn 0}%
\special{sh 1.000}%
\special{ia 3200 800 30 30  0.0000000 6.2831853}%
}}%
{\color[named]{Black}{%
\special{pn 8}%
\special{ar 3200 800 30 30  0.0000000 6.2831853}%
}}%
%
{\color[named]{Black}{%
\special{pn 0}%
\special{sh 1.000}%
\special{ia 2800 1400 30 30  0.0000000 6.2831853}%
}}%
{\color[named]{Black}{%
\special{pn 8}%
\special{ar 2800 1400 30 30  0.0000000 6.2831853}%
}}%
%
{\color[named]{Black}{%
\special{pn 8}%
\special{ar 2800 1000 30 30  0.0000000 6.2831853}%
}}%
%
{\color[named]{Black}{%
\special{pn 8}%
\special{pa 2780 990}%
\special{pa 2400 800}%
\special{dt 0.045}%
}}%
%
{\color[named]{Black}{%
\special{pn 8}%
\special{pa 2820 990}%
\special{pa 3200 800}%
\special{dt 0.045}%
}}%
%
{\color[named]{Black}{%
\special{pn 8}%
\special{pa 2800 1030}%
\special{pa 2800 1400}%
\special{dt 0.045}%
}}%
\put(26.0000,-9.0000){\makebox(0,0)[rt]{$1$}}%
\put(30.0000,-9.0000){\makebox(0,0)[lt]{$2$}}%
\put(27.5000,-12.0000){\makebox(0,0){$3$}}%
\put(8.0000,-8.0000){\makebox(0,0)[rb]{$(1)$}}%
\put(8.0000,-14.0000){\makebox(0,0)[rt]{$(2)$}}%
\put(14.0000,-14.0000){\makebox(0,0)[lt]{$(3)$}}%
\put(14.0000,-8.0000){\makebox(0,0)[lb]{$(4)$}}%
\put(23.8000,-8.0000){\makebox(0,0)[rb]{$(1)$}}%
\put(32.2000,-8.0000){\makebox(0,0)[lb]{$(2)$}}%
\put(27.8000,-14.0000){\makebox(0,0)[rt]{$(3)$}}%
\end{picture}%
\caption{Graphs $X$ and $Y$}
\label{fig:H}
\end{figure}
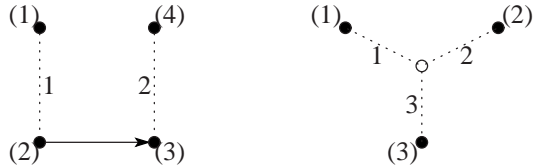
The numbers assigned to vertices and edges indicate the ordering.
Denote by $\D^{p,q}$ the vector space spanned by oriented graphs of \emph{order} $p$ and of \emph{degree} $q$,
where
\[
 \mathrm{ord}(\Gamma):=\abs{E_{\theta}(\Gamma)}-\abs{E_{\rm e}(\Gamma)},\quad
 \deg(\Gamma):=2\abs{E_{\theta}(\Gamma)}-3\abs{E_{\rm e}(\Gamma)}-\abs{E_{\rm i}(\Gamma)}.
\]
For an oriented graph $\Gamma$ consider the \emph{configuration space}
\[
 C_{\Gamma}^{\circ}:=\{(f;\bfx;\bfy)\in\emb{n}{j}\times\Conf^{\circ}_{\abs{V_{\rm i}(\Gamma)}}(\R^j)\times\Conf^{\circ}_{\abs{V_{\rm e}(\Gamma)}}(\R^n)\mid
 f(x_i)\ne y_j\ (\forall i,j)\},
\]
where $\Conf^{\circ}_m(M):=M^{\times m}\setminus\bigcup_{i<j}\{x_i\ne x_j\}$ stands for the usual configuration space.
To each $\eta$-edge $\overrightarrow{ij}$ and $\theta$-edge $\overrightarrow{st}$, we assign the \emph{generalized Gauss maps}
\begin{alignat*}{10}
 &\varphi^{\eta}_{ij}  \colon C^{\circ}_{\Gamma}\to S^{j-1},& &\quad& &(f;\bfx;\bfy)\mapsto\frac{\bfx_j-\bfx_i}{\abs{\bfx_j-\bfx_i}},&\\
 &\varphi^{\theta}_{st}\colon C^{\circ}_{\Gamma}\to S^{n-1},& &\quad& &(f;\bfx;\bfy)\mapsto\frac{\bfz_t-\bfz_s}{\abs{\bfz_t-\bfz_s}},&
\end{alignat*}
where
\[
 z_s:=
 \begin{cases}
  f(x_s) & \text{if }s\text{ is an i-vertex},\\
  y_s    & \text{if }s\text{ is an e-vertex}.
 \end{cases}
\]
Choose a representative of the orientation of $\Gamma$ so that $\theta$-edges follow after $\eta$-edges.
Let $\varphi_{\Gamma}\colon C_{\Gamma}^{\circ}\to(S^{j-1})^{\times\abs{E_{\eta}(\Gamma)}}\times(S^{n-1})^{\times\abs{E_{\theta}(\Gamma)}}$ be the product of all these Gauss maps assigned to the edges of $\Gamma$.
The product is taken in the order of the orientation of $\Gamma$.
Let $\vol_{S^{N-1}}\in\Omega^{N-1}_{\mathit{DR}}(S^{N-1})$ ($N=n,j$) be a unit volume form of $S^{N-1}$ that is (anti-)invariant under the action of $O(N)$ fixing the poles $\{\pm e_N\}\subset S^{N-1}$, where $e_N:=(0,\dotsc,0,1)\in\R^N$.
Define $\omega_{\Gamma}\in\Omega^*_{DR}(C_{\Gamma}^{\circ})$ by
\[
 \omega_{\Gamma}:=\varphi^*_{\Gamma}(\vol_{S^{j-1}}^{\times\abs{E_{\eta}(\Gamma)}}\times\vol_{S^{j-1}}^{\times\abs{E_{\theta}(\Gamma)}}),
\]
where $\vol_{S^{a-1}}\times\vol_{S^{b-1}}$ is the product of volume forms pulled back on $S^{a-1}\times S^{b-1}$ by the projections.
Integrating $\omega_{\Gamma}$ along the fibers of the natural projections
\[
 \pi_{\Gamma}\colon C^{\circ}_{\Gamma}\to\emb{n}{j}
\]
(whose fibers are subspaces of usual configuration spaces), we obtain
\[
 I(\Gamma):=\pi_{\Gamma*}\omega_{\Gamma}\in\Omega^*_{\mathit{DR}}(\emb{n}{j}).
\]
$I(\Gamma)$ is independent of the representative of the orientation of $\Gamma$.
The degree of $I(\Gamma)$ can be given using ${\rm ord}(\Gamma)$, $\deg(\Gamma)$, and the first betti number of $\Gamma$ (thought of as a $1$-dimensional cell complex); see \cite[\S3]{K08}.

\begin{remark}\label{rem:compact}
The above integrals converge since we may replace $C^{\circ}_{\Gamma}$ with its (fiberwise) Fulton--MacPherson compactification (see \cite{Sinha03}) denoted by $C_{\Gamma}$, over which the generalized Gauss maps are smoothly extended.
Thus we obtain a linear map $I\colon \D^*\to\Omega^*_{DR}(\emb{n}{j})$.
\end{remark}

By the generalized Stokes' theorem, $dI(\Gamma)$ is a linear combination of integrals along the codimension $1$ boundary faces of the fibers of $\pi_{\Gamma}$, which are subspaces of compactified configuration spaces (Remark~\ref{rem:compact}).
The boundary faces of the compactifications are stratified according to the ``complexity of collisions of points.''
The strata in which exactly two points collide are called \emph{principal}.
The author proved in \cite{K08} that, if the volume forms are (anti-)invariant, many such integrals along the non-principal boundary faces cancel out or vanish; the large part of the proof follows the arguments in \cite{CCL02,CattaneoRossi05,Watanabe07}.
Thus if we define the \emph{coboundary maps} $\delta\colon\D^{p,q}\to\D^{p,q+1}$ as the signed sum of graphs obtained by collapsing two vertices together with the edges between them (see \cite[\S2.3]{K08} for signs), the map $I$ becomes a cochain map.
More precisely, the following holds.

\begin{theorem}[{\cite[Theorem~1.2]{K08}}]\label{thm:K08}
The map $I$ is a cochain map if
\begin{itemize}
\item
	restricted to the subcomplex of tree graphs and $n,j$ are of same parity, or
\item
	restricted to the subcomplex of graphs of betti number not greater than $1$ and both $n,j$ are odd.
\end{itemize}
\end{theorem}

Conjecturally, the map $I$ would always be a cochain map and a quasi-isomorphism.
In fact, $\D^*$ looks very similar to the graph complexes given in \cite{AroneTurchin14} that compute the rational homology and homotopy of $\emb{n}{j}$ in the stable dimensions.
The map $I$ yields many non-zero cohomology classes even in the non-stable dimensions and in some dimensions not necessarily satisfying the condition in Theorem~\ref{thm:K08}; for example, $H^3_{DR}(\emb{5}{2})\ne 0$.
See \cite{KWatanabe12,Watanabe07}.

\subsection{The Haefliger invariant}\label{ss:Haefliger_inv}
Let $X,Y\in\D^{2,0}$ be graphs in Figure~\ref{fig:H}, and denote $C_{4,0}:=C_X$ and $C_{3,1}:=C_Y$.
We note that $C_{4,0}=\emb{6k}{4k-1}\times\Conf^{\circ}_4(\R^{4k-1})$ and $C_{3,1}\subset\emb{6k}{4k-1}\times\Conf^{\circ}_3(\R^{4k-1})\times\R^{6k}$.
In fact, $H:=X/2+Y/6\in\D^{2,0}$ is a cocycle and $I(H)$ is a $0$-form of $\emb{6k}{4k-1}$.
Unfortunately, Theorem~\ref{thm:K08} might fail for $(n,j)=(6k,4k-1)$;
at present, it is not known whether the integral along the ``anomalous boundary face'' $\Sigma_{3,1}\subset\partial C_{3,1}$ (where all the four points $f(\bfx_1),f(\bfx_2),f(\bfx_3)$, and $\bfx_4$ collapse to a single point) vanishes or not.
In \cite{K08} we add a correction term $c$ (defined below) to $I(H)$ to kill the anomalous contribution and get a closed-form $\calH:=I(H)+c$.

The correction term $c$ is defined as follows.
The interior ${\rm Int}\,\Sigma_{3,1}$ of $\Sigma_{3,1}$ can be described by the following pullback square:
\[
 \xymatrix{
  {\rm Int}\,\Sigma_{3,1}\ar[r]\ar[d] & B\ar[d]^-{\rho}\\
  \R^{4k-1}\times\emb{6k}{4k-1}\ar[r]^-D & \Inj_{6k,4k-1}
 }
\]
Let us explain the spaces and maps in the above diagram.
$\Inj_{6k,4k-1}$ is the space of linear, injective maps $\R^{4k-1}\hookrightarrow\R^{6k}$.
The space $B$ is defined as
\begin{multline*}
 B:=\{(\lambda;(\bfx_1,\bfx_2,\bfx_3);\bfx_4)\in\Inj_{6k,4k-1}\times\Conf^{\circ}_3(\R^{4k-1})\times\R^{6k}\mid \\
 \lambda(\bfx_i)\ne\bfx_4,\, 1\le i\le 3\}/\R^1_+\ltimes\R^{4k-1},
\end{multline*}
where $\R^1_+\ltimes\R^{4k-1}$ acts diagonally on $\Conf^{\circ}_3(\R^{4k-1})\times\R^{6k}$ as the positive scalings and translations along $\lambda (\R^{4k-1})$.
The map $\rho$ is the natural projection, and $D$ is the differential
\[
 D(x;f):=(\mathit{df}_x\colon T_x\R^{4k-1}=\R^{4k-1}\hookrightarrow\R^{6k}=T_{f(x)}\R^{6k}).
\]
For $i=1,2,3$, the map
\[
 \psi_i\colon B\to S^{6k-1},\quad
 [\iota;(\bfx_1,\bfx_2,\bfx_3);\bfx_4]\mapsto\frac{\bfx_4-\iota(\bfx_i)}{\abs{\bfx_4-\iota(\bfx_i)}}
\]
is well defined.
Put $\psi:=\psi_1\times\psi_2\times\psi_3\colon B\to (S^{6k-1})^{\times 3}$ and consider
\[
 \omega :=\psi^*\vol_{S^{6k-1}}^{\times 3}\in\Omega^{18k-3}_{\mathit{DR}}(B).
\]
We can see that $\rho_*\omega\in\Omega^{4k}_{\mathit{DR}}(\Inj_{6k,4k-1})$ is closed \cite[Lemma~5.22]{K08}.
$\Inj_{6k,4k-1}$ is homeomorphic to the Stiefel manifold $V_{6k,4k-1}$ of $(4k-1)$-frames in $\R^{6k}$, and hence $H^{4k}_{\mathit{DR}}(\Inj_{6k,4k-1})=0$.
Therefore, we can find $\mu\in\Omega^{4k-1}_{\mathit{DR}}(\Inj_{6k,4k-1})$ such that
\[
 \rho_*\omega =d\mu.
\]
The correction term $c\colon \emb{6k}{4k-1}\to\R$ is defined by
\[
 c(f):=\frac{1}{6}\int_{\R^{4k-1}}(\mathit{df})^*\mu\in\R,
\]
where $\mathit{df}\colon\R^{4k-1}\to\Inj_{6k,4k-1}$ is defined by the differential $x\mapsto\mathit{df}_x$.

Let $C_{4,0}(f)$ and $C_{3,1}(f)$ be the fibers of $\pi_X\colon C_{4,0}\to\emb{6k}{4k-1}$ and $\pi_Y\colon C_{3,1}\to\emb{6k}{4k-1}$ over $f$.
These fibers are finite-dimensional configuration spaces with $\dim C_{4,0}(f)=\deg\omega_X$ and $\dim C_{3,1}(f)=\deg\omega_Y$.
$\calH(f)$ is calculated as
\[
 \calH(f)=\frac{1}{2}\int_{C_{4,0}(f)}\omega_X+\frac{1}{6}\int_{C_{3,1}(f)}\omega_Y+c(f).
\]

\begin{theorem}[\cite{K08}]\label{thm:H}
If $\vol_{S^{N-1}}$ ($N=6k,4k-1$) are (anti-)invariant, then the $0$-form $\calH:=I(H)+c\colon\emb{6k}{4k-1}\to\R$ is closed and induces a group isomorphism $\pi_0(\emb{6k}{4k-1})\cong\Z$.
In particular, $\calH$ is a $\Z$-valued invariant.
\end{theorem}

In \cite{K08} Theorem~\ref{thm:H} is proved by evaluating $\calH$ over a generator of $\pi_0(\emb{6k}{4k-1})$ given by the spinning construction \cite{Budney08,RosemanTakase07}.
The computation in \S\ref{s:example} gives an alternative proof.

To simplify the computations below, we will take the (anti-)invariant volume forms $\vol_{S^{N-1}}$ ($N=6k,4k-1$) so that their supports are contained in small neighborhoods of the poles $\pm e_N:=(0,\dotsc,0,\pm 1)\in\R^N$.
We call such a volume form a \emph{Dirac-type} volume form.
The following allows us to use such volume forms.

\begin{proposition}[{\cite[Propositions~3.5, 3.6]{K08}}]\label{prop:vol}
The invariant $\calH$ is independent of the choice of the (anti-)invariant volume forms $\vol_{S^{N-1}}$, $N=6k,4k-1$.
\end{proposition}
\section{Proofs of Theorems~\ref{thm:main1} and \ref{thm:order2}}\label{s:finite}
It is well known that the \emph{linking number} of closed oriented submanifolds $M^{2k-1}\sqcup N^{2k-1}\subset\R^{4k-1}$ is an isotopy invariant and can be defined as
\[
 lk(M,N):=\int_{M\times N}\varphi^*\vol_{S^{4k-2}},
\]
where $\varphi\colon M\times N\to S^{4k-2}$ is the generalized Gauss map given by
\[
 \varphi(x,y):=\frac{y-x}{\abs{y-x}}.
\]
If $M,N$ are in generic positions, then $p(M\sqcup N)$ is generically immersed in $\R^{4k-2}$ and the pairs $(x,y)\in M\times N$ such that $p(\varphi(x,y))=\zero\in\R^{4k-2}$ form a $0$-dimensional submanifold of $M\times N$.
If, moreover, $\vol_{S^{4k-2}}$ is Dirac-type (see the paragraph before Proposition~\ref{prop:vol}), then
\begin{equation}\label{eq:linking_number}
\begin{split}
 lk(M,N)&=\sum_{(x,y)\in(p\circ\varphi)^{-1}(\zero)}\int_{\genfrac{}{}{0pt}{1}{\text{a neighborhood}}{\text{of }(x,y)}}\varphi^*\vol_{S^{4k-2}} \\
 &=\frac{1}{2}\sum_{(x,y)\in(p\circ\varphi)^{-1}(\zero)}\deg\varphi|_{\genfrac{}{}{0pt}{1}{\text{a neighborhood}}{\text{of }(x,y)}}
\end{split}
\end{equation}
and each $\deg\varphi=\pm1$.
This is because the integrand $\varphi^*\vol_{S^{4k-2}}$ is zero outside neighborhoods of such pairs, and the integral of $\vol_{S^{4k-2}}$ over one component of $\supp(\vol_{S^{4k-2}})$ is $1/2$.
This interpretation gives us the following, which we use in \S\ref{s:gen_imm}.

\begin{lemma}\label{lem:link_connected_sum}
Let $M_1,M_2$ and $N_1,N_2$ be $(2k-1)$-dimensional disjoint oriented submanifolds of $\R^{4k-1}$.
If the connected-sums $M_1\sharp M_2$ and $N_1\sharp N_2$ are taken in $\R^{4k-1}$ so that $\abs{p(M_i)\cap p(N_j)}$ ($i,j=1,2$) do not increase, then
\[
 \sum_{i,j=1,2}lk(M_i,N_j)=lk(M_1\sharp M_2,N_1\sharp N_2).
\]
\end{lemma}

For a single oriented submanifold $L^{2k-1}\subset\R^{4k-1}$ a similar formula to \eqref{eq:linking_number} does not give rise to an isotopy invariant, but if $L$ is almost planar, then such a function can be computed by counting $(x,y)\in\Conf^{\circ}_2(L)$ with $p(\varphi(x,y))=\zero$.

\begin{definition}\label{def:writhe}
Let $L^{2k-1}\subset\R^{4k-1}$ be a generic closed oriented submanifold such that $p(L)\subset\R^{4k-2}$ is a generically immersed manifold.
Define the \emph{writhe} $w(L)$ of $L$ by
\begin{equation}\label{eq:writhe_integral}
\begin{split}
 w(L)&=\sum_{(x,y)\in(p\circ\varphi)^{-1}(\zero)}\int_{\genfrac{}{}{0pt}{1}{\text{a neighborhood}}{\text{of }(x,y)}}\varphi^*\vol_{S^{4k-2}}\\
 &=\frac{1}{2}\sum_{(x,y)\in(p\circ\varphi)^{-1}(\zero)}\deg\varphi|_{\genfrac{}{}{0pt}{1}{\text{a neighborhood}}{\text{of }(x,y)}}.
\end{split}
\end{equation}
\end{definition}

For an almost planar $f\in\emb{6k}{4k-1}$, denote by $f^{\delta}$ the embedding obtained by a scaling in the $x_{6k}$-direction so that $f^{\delta}(\R^{4k-1})\subset\R^{6k-1}\times[0,\delta]$ (we often abbreviate $f^{\delta}$ as $f$).
$L^{\epsilon}_i$'s remain unchanged for any $\delta$.
We compute $I(X)(f^{\delta})$, $I(Y)(f^{\delta})$ and $c(f^{\delta})$ in the limit $\delta\to 0$.

\begin{proposition}[cf.\ {\cite[Proposition~4.3]{LinWang96}}]\label{prop:I(X)}
Suppose $\vol_{S^{N-1}}$ ($N=6k,4k-1$) are Dirac-type.
If $f\in\emb{6k}{4k-1}$ is almost planar and generic so that $lk(L_i^{\epsilon},L_j^{\epsilon'})$ and $w(L_i^{\epsilon})$ can be calculated by \eqref{eq:linking_number} and \eqref{eq:writhe_integral}, then
\[
 \lim_{\delta\to 0}I(X)(f^{\delta})=
 \frac{1}{2}\sum_{(i,\epsilon)<(j,\epsilon')}(-1)^{\epsilon+\epsilon'}lk(L_i^{\epsilon},L_j^{\epsilon'})
 +\frac{1}{4}\sum_{i,\epsilon}w(L_i^{\epsilon}).
\]
\end{proposition}

\begin{proof}
A configuration $\vec{\xi}=(\bfxi_1,\dotsc,\bfxi_4)\in C_{4,0}(f)$ can non-trivially contribute to $I(X)$ only if $\vec{\xi}\in\varphi_X^{-1}(\supp(\vol_{S^{6k-1}}^{\times 2}\times\vol_{S^{4k-2}}))$.
Since $\vol_{S^{N-1}}$ are Dirac-type, such a $\vec{\xi}$ must be in a neighborhood of $\varphi_X^{-1}(\pm e_{6k},\pm e_{6k},\pm e_{4k-1})$, where $e_N:=(0,\dotsc,0,1)\in S^{N-1}$.

If $\delta$ is sufficiently close to $0$, then no vectors tangent to $f(\R^{4k-1})$ point $\supp(\vol_{S^{6k-1}})$.
Thus $\vec{\xi}$ can be in $\varphi_X^{-1}(\supp(\vol_{S^{6k-1}}^{\times 2}\times\vol_{S^{4k-2}}))$ in the limit $\delta\to 0$ only if $(\bfxi_1,\bfxi_2)\in N(L_i^{\epsilon})\times N(L_i^{\epsilon+1})$ and $(\bfxi_3,\bfxi_4)\in N(L_j^{\epsilon'})\times N(L_j^{\epsilon'+1})$ for some $i,j$, possibly $i=j$ (recall that $N(L_i^{\epsilon})\subset\R^{4k-1}$ are closed disjoint tubular neighborhoods of $L_i^{\epsilon}$).
For $(s,t)=(1,2),(4,3)$, and any $\bfxi_t\in N(L_i^{\epsilon})$, we always find $\bfxi_s\in N(L_i^{\epsilon+1})$ such that $p(f(\bfxi_s)-f(\bfxi_t))$ is close to $\zero$.
Therefore, finding such a $\vec{\xi}=(\bfxi_1,\dotsc,\bfxi_4)\in\varphi_X^{-1}(\pm e_{6k},\pm e_{6k},\pm e_{4k-1})$ is equivalent to finding $(\bfxi_2,\bfxi_3)$ satisfying $p(\varphi^{\eta}_{23}(\bfxi_2,\bfxi_3))=\zero$.
By our assumption on $f$, the set of such $\vec{\xi}$\,'s is a $0$-dimensional submanifold of $C_{4,0}(f)$, each component of whose neighborhood is mapped homeomorphically into a component of $\supp(\vol_{S^{6k-1}}^{\times 2}\times\vol_{S^{4k-2}})$ via $\varphi_X$.
The integral of $\omega_X$ over such a component is $\pm(1/2)^3$, where the sign is the local degree of $\varphi_X$ at $\vec{\xi}$---that is, the determinant of the Jacobian $J(\varphi_X)_{\vec{\xi}}$.
The sum of these degrees would amount to linking numbers by \eqref{eq:linking_number}.

To compute $\deg\varphi_X$ at each $\vec{\xi}\in\varphi_X^{-1}(\pm e_{6k},\pm e_{6k},\pm e_{4k-1})$, we recall from \cite{Ekholm01,Ekholm01-2} the local model for two-fold self-intersection.
Let $g\in\Imm{6k-1}{4k-1}$ be a generic immersion, and let $q=g(p_1)=g(p_2)$ be a transverse two-fold self-intersection point.
In some local coordinates centered at $p_1$, $p_2$, and $q$, $g$ is given by
\begin{alignat*}{10}
 g(x_1,\dots,x_{4k-1})&=(x_1,\,& \dotsc,\, &x_{2k-1},\,& &x_{2k},\,& \dotsc,\, &x_{4k-1},\,& &0,       & \dotsc,\, &0)\quad       & &\text{near }p_1,&\\
 g(y_1,\dots,y_{4k-1})&=(y_1,\,& \dotsc,\, &y_{2k-1},\,& &0,       & \dotsc,\, & 0,        & &y_{2k},\,& \dotsc,\, &y_{4k-1})\quad& &\text{near }p_2.&
\end{alignat*}
Now consider a configuration $\vec{\xi}\in C_{4,0}(f)$ such that
\begin{equation}\label{eq:configuration1}
 \begin{cases}
  \bfxi_1\in L_i^1,\ \bfxi_2\in L_i^0,\ \bfxi_3\in L_j^0\text{ and }\bfxi_4\in L_j^1, \\
  \varphi_X(\vec{\xi})=(-e_{6k},-e_{6k},e_{4k-1})
 \end{cases}
\end{equation}
(Figure~\ref{fig:conf5}, left).
\begin{figure}[htb]
\centering
\includegraphics{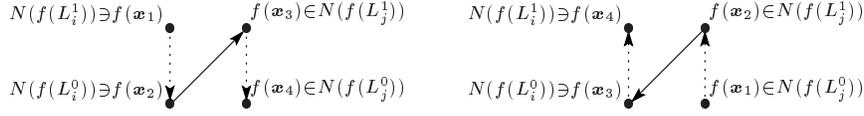}
\caption{Two configurations with the same contributions to $lk(L_i^0,L_j^1)$; the left shows a neighborhood of $\vec{\xi}$ satisfying \eqref{eq:configuration1}, and the right shows a neighborhood of $\vec{\xi'}$ satisfying \eqref{eq:configuration2}.}
\label{fig:conf5}
\end{figure}
Suppose that $(\bfxi_2,\bfxi_3)$ is a positive crossing---that is, $\deg\varphi_{23}^{\eta}|_{(\bfxi_2,\bfxi_3)}=+1$.
Then we can choose some local coordinates $\bfx$, $\bfy$, and $\bfz$ centered at $\bfxi_1$, $\bfxi_2$, and $\bfxi_4$ such that $\bfxi_3=(0,\dotsc,0,1)$ in the $\bfy$-coordinate (the same coordinate as for $\bfxi_2$), and we can also choose local coordinates in $\R^{6k}$ in which $f$ is given by
\begin{alignat*}{10}
 f(\bfx)&=(x_1,\,& \dotsc,\, &x_{2k-1},\,& &x_{2k},\,& \dotsc,\, &x_{4k-1},\,& &0,       & \dotsc,\, &0,\,& &1&)\quad       & &\text{near }\bfxi_1,&\\
 f(\bfy)&=(y_1,\,& \dotsc,\, &y_{2k-1},\,& &0,       & \dotsc,\, & 0,        & &y_{2k},\,& \dotsc,\, &y_{4k-1},\,& &0&)\quad& &\text{near }\bfxi_2,&
\end{alignat*}
and
\begin{alignat*}{100}
 f(\bfy)&=(0,\,  & \dotsc,\, &0,\,       & &y_1,& \dotsc,\, &y_{2k-1},\,& &y_{2k},\,& \dotsc,\, &y_{4k-2},\,& &y_{4k-1},\,& &0,         & &1)\quad& &\text{near }\bfxi_3,&\\
 f(\bfz)&=(z_1,\,& \dotsc,\, &z_{2k-1},\,& &0,  & \dotsc,\, &0,\,       & &z_{2k},\,& \dotsc,\, &z_{4k-2},\,& &1,\,       & &z_{4k-1},\,& &0)\quad& &\text{near }\bfxi_4.&
\end{alignat*}
Then by Lemma~\ref{lem:orientation}, as oriented manifolds, $L_*^{\epsilon}$'s are given by
\begin{itemize}
\item
 $L_i^1\cap(\bfx\text{-coordinate})=+\R^{2k-1}\times\{\zero\}^{2k}$,
\item
 $L_i^0\cap(\bfy\text{-coordinate})=+\R^{2k-1}\times\{\zero\}^{2k}$,
\item
 $L_j^1\cap(\bfy\text{-coordinate})=\{\zero\}^{2k-1}\times(-\R^{2k-1})\times\{1\}$, and
\item
 $L_j^0\cap(\bfz\text{-coordinate})=\{\zero\}^{2k-1}\times(-\R^{2k-1})\times\{0\}$
\end{itemize}
(see Figure~\ref{fig:conf}).
\begin{figure}[htb]
\centering
\includegraphics[scale=0.8]{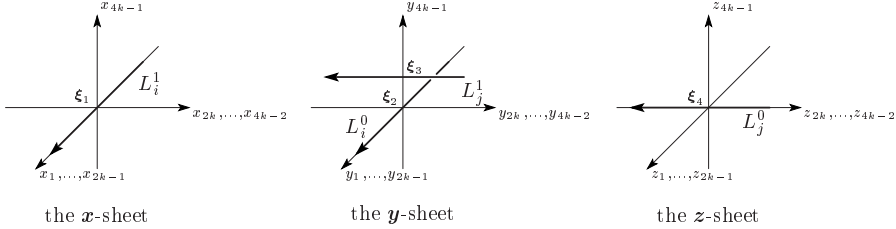}
\caption{Local picture of $L_i^0\cup L_i^1\cup L_j^0\cup L_j^1$}
\label{fig:conf}
\end{figure}
Using this local model, we can compute the Jacobian $J(\varphi_X)_{\vec{\xi}}$ of
\[
 \varphi_X\colon  N(L_i^1)\times N(L_i^0)\times N(L_j^1)\times N(L_j^0)\to S^{6k-1}\times S^{6k-1}\times S^{4k-2}
\]
at $\vec{\xi}$ explicitly.
Let $e_i:=(0,\dots,0,1,0,\dotsc,0)$ be the $i$th unit vector.
With respect to the natural positive basis $e_1,\dotsc,e_{16k-4}$ of $T_{\vec{\xi}}\Conf_4(\R^{4k-1})\cong T_{\vec{\xi}}\R^{16k-4}$ and the natural positive bases of the tangent spaces of spheres
\[
 e_1,\dotsc,e_{6k-1}\in T_{-e_{6k}}S^{6k-1}\quad\text{and}\quad
 e_1,\dotsc,e_{4k-2}\in T_{e_{4k-1}}S^{4k-2},
\]
following the ``outward normal first'' convention, $J(\varphi_X)_{\vec{\xi}}$ is given as in Appendix~\ref{s:Jacobian} and its determinant is $-1$.

A local model for a negative crossing $(\bfxi_2,\bfxi_3)$ (namely, $\deg\varphi_{23}^{\eta}|_{(\bfxi_2,\bfxi_3)}=-1$) is obtained from the above model by reversing the orientation of the $\bfz$-sheet, and, in this case, $J(\varphi_X)_{\vec{\xi}}=+1$.
Thus the integral of $\omega_X$ over a neighborhood of $\vec{\xi}$ satisfying \eqref{eq:configuration1} with $\deg\varphi^{\eta}_{23}|_{(\bfxi_2,\bfxi_3)}=\pm1$ is $\mp(1/2)^3$, and by \eqref{eq:linking_number} their sum is equal to $-lk(L_i^0,L_j^1)/4$.
By symmetry of $X$, the configurations near $\vec{\xi'}$ satisfying
\begin{equation}\label{eq:configuration2}
 \begin{cases}
  \bfxi'_1\in L_j^0,\ \bfxi'_2\in L_j^1,\ \bfxi'_3\in L_i^1\text{ and }\bfxi'_4\in L_i^0, \\
  \varphi_X(\vec{\xi})=(e_{6k},e_{6k},-e_{4k-1})
 \end{cases}
\end{equation}
(see Figure~\ref{fig:conf5}, right) also contribute to $I(X)(f)$ by $-lk(L_i^0,L_j^1)/4$.
Thus a link $L_i^0\sqcup L_j^1$ contributes to $I(X)(f)$ by $-lk(L_i^0,L_j^1)/2$.

A similar computation shows that $L_i^0$ and $L_j^0$ contribute to $I(X)(f)$ by $+lk(L_i^0,L_j^0)/2$;
we have the same Jacobian matrix as above, but in this case $\varphi^{\theta}_{34}(\vec{\xi})=e_{6k}$ and $(e_1,\dotsc,e_{6k-1})$ represent the negative orientation of $T_{e_{6k}}S^{6k-1}$ and the sign of the degree changes.
This observation shows that, in general, the link $L_i^{\epsilon}\sqcup L_j^{\epsilon'}$, $(i,\epsilon)\ne(j,\epsilon')$, contributes to $I(X)(f)$ by $(-1)^{\epsilon+\epsilon'}lk(L_i^{\epsilon},L_j^{\epsilon'})/2$.

In the case $(i,\epsilon)=(j,\epsilon')$, if $f$ is generic so that $w(L_i^{\epsilon})$ can be calculated by \eqref{eq:writhe_integral}, then the same computation as above shows that the configurations in $N(L_i^{\epsilon+1})\times N(L_i^{\epsilon})\times N(L_i^{\epsilon})\times N(L_i^{\epsilon+1})$ contributes to $I(X)(f)$ by $+w(L_i^{\epsilon})/4$ (no sign appears in this case since $\epsilon=\epsilon'$).
\end{proof}

\begin{lemma}\label{lem:I(Y)}
Let $f\in\emb{6k}{4k-1}$ be almost planar, and suppose that $\vol_{S^{6k-1}}$ is Dirac-type.
Then
\[
 \lim_{\delta\to 0}I(Y)(f^{\delta})=\lim_{\delta\to 0}I(Y)(f^{\delta}_S)
\quad\text{and}\quad
 \lim_{\delta\to 0}c(f^{\delta})=\lim_{\delta\to 0}c(f^{\delta}_S).
\]
\end{lemma}

\begin{proof}
The function $c\colon\emb{6k}{4k-1}\to\R$ is defined and continuous on the space of immersions, because $c(f)$ is determined by the differential of $f$ (see \S\ref{ss:Haefliger_inv}).
Thus $c(f^{\delta})$ continuously depends on $\delta$ including $\delta=0$ (then $f$ and $f_S$ collapse down to an immersion $f^0=f^0_S=p\circ f$), and hence $\lim_{\delta\to 0}c(f)=c(p\circ f)=\lim_{\delta\to 0}c(f_S)$.

The above observation also implies that $\lim_{\delta\to 0}I(Y)(f^{\delta})$ exists, because $I(Y)(f^{\delta})=6\calH(f^{\delta})-3I(X)(f^{\delta})-6c(f^{\delta})$ and the limit of the right-hand side exists by Proposition~\ref{prop:I(X)}, the existence of $\lim c(f^{\delta})$, and the fact that $\calH$ is an isotopy invariant.
To compute $\lim_{\delta\to 0}I(Y)(f^{\delta})$, we may assume that $\delta<1$.
Choose open neighborhoods $f(N(L_i^0))\subset U_i'\subset V_i'$ in $\R^{6k-1}$ (we are assuming $f$ is almost planar) so that $\overline{U_i'}\subset V_i'$, $V_i'\cap V_j'=\emptyset$ ($i\ne j$), and $f(N(L_i^1))\subset U_i:=U_i'\times[-1,1]$ (Figure~\ref{fig:conf4}).
\begin{figure}[htb]
\centering
\unitlength 0.1in
\begin{picture}( 24.0000, 13.6500)(  4.0000,-16.0000)
{\color[named]{Black}{%
\special{pn 13}%
\special{pa 1200 1000}%
\special{pa 1220 1000}%
\special{pa 1246 994}%
\special{pa 1250 992}%
\special{pa 1256 992}%
\special{pa 1286 980}%
\special{pa 1296 974}%
\special{pa 1300 972}%
\special{pa 1326 956}%
\special{pa 1330 952}%
\special{pa 1336 950}%
\special{pa 1340 946}%
\special{pa 1346 942}%
\special{pa 1350 938}%
\special{pa 1356 936}%
\special{pa 1370 924}%
\special{pa 1376 920}%
\special{pa 1426 880}%
\special{pa 1430 878}%
\special{pa 1446 866}%
\special{pa 1450 862}%
\special{pa 1456 858}%
\special{pa 1460 856}%
\special{pa 1466 852}%
\special{pa 1470 848}%
\special{pa 1476 844}%
\special{pa 1500 830}%
\special{pa 1506 828}%
\special{pa 1516 822}%
\special{pa 1546 810}%
\special{pa 1550 808}%
\special{pa 1556 806}%
\special{pa 1580 802}%
\special{pa 1586 802}%
\special{pa 1590 800}%
\special{pa 1610 800}%
\special{pa 1616 802}%
\special{pa 1620 802}%
\special{pa 1646 806}%
\special{pa 1650 808}%
\special{pa 1656 810}%
\special{pa 1686 822}%
\special{pa 1696 828}%
\special{pa 1700 830}%
\special{pa 1726 844}%
\special{pa 1730 848}%
\special{pa 1736 852}%
\special{pa 1740 856}%
\special{pa 1746 858}%
\special{pa 1750 862}%
\special{pa 1756 866}%
\special{pa 1770 878}%
\special{pa 1776 880}%
\special{pa 1826 920}%
\special{pa 1830 924}%
\special{pa 1846 936}%
\special{pa 1850 938}%
\special{pa 1856 942}%
\special{pa 1860 946}%
\special{pa 1866 950}%
\special{pa 1870 952}%
\special{pa 1900 972}%
\special{pa 1906 974}%
\special{pa 1916 980}%
\special{pa 1946 992}%
\special{pa 1950 992}%
\special{pa 1956 994}%
\special{pa 1980 1000}%
\special{pa 2000 1000}%
\special{fp}%
}}%
%
{\color[named]{Black}{%
\special{pn 13}%
\special{pa 800 1000}%
\special{pa 1200 1000}%
\special{fp}%
}}%
%
{\color[named]{Black}{%
\special{pn 13}%
\special{pa 2000 1000}%
\special{pa 2400 1000}%
\special{fp}%
}}%
%
{\color[named]{Black}{%
\special{pn 13}%
\special{pa 1800 1200}%
\special{pa 1480 880}%
\special{fp}%
}}%
%
{\color[named]{Black}{%
\special{pn 13}%
\special{pa 1300 700}%
\special{pa 1430 830}%
\special{dt 0.045}%
}}%
%
{\color[named]{Black}{%
\special{pn 4}%
\special{pa 2000 780}%
\special{pa 2000 1400}%
\special{fp}%
}}%
%
{\color[named]{Black}{%
\special{pn 8}%
\special{ar 1600 600 400 100  0.0000000 6.2831853}%
}}%
%
{\color[named]{Black}{%
\special{pn 4}%
\special{ar 1600 1400 400 100  6.2831853 6.2831853}%
\special{ar 1600 1400 400 100  0.0000000 3.1415927}%
}}%
\put(14.0000,-11.0000){\makebox(0,0)[rt]{$\sb{U_i}$}}%
\put(9.8000,-13.0000){\makebox(0,0)[rt]{$\sb{V_i}$}}%
\put(18.0000,-9.0000){\makebox(0,0)[lb]{$\sb{f(N(L_i^1))}$}}%
\put(18.0000,-12.0000){\makebox(0,0)[lb]{$\sb{f(N(L_i^0))}$}}%
\put(16.0000,-14.0000){\makebox(0,0){$C_{3,1}^{(1)}$}}%
%
{\color[named]{Black}{%
\special{pn 4}%
\special{pa 400 400}%
\special{pa 780 780}%
\special{fp}%
}}%
%
{\color[named]{Black}{%
\special{pn 4}%
\special{pa 2400 400}%
\special{pa 2780 780}%
\special{fp}%
}}%
%
{\color[named]{Black}{%
\special{pn 4}%
\special{ar 1600 1400 400 100  3.1415927 3.1895927}%
\special{ar 1600 1400 400 100  3.3335927 3.3815927}%
\special{ar 1600 1400 400 100  3.5255927 3.5735927}%
\special{ar 1600 1400 400 100  3.7175927 3.7655927}%
\special{ar 1600 1400 400 100  3.9095927 3.9575927}%
\special{ar 1600 1400 400 100  4.1015927 4.1495927}%
\special{ar 1600 1400 400 100  4.2935927 4.3415927}%
\special{ar 1600 1400 400 100  4.4855927 4.5335927}%
\special{ar 1600 1400 400 100  4.6775927 4.7255927}%
\special{ar 1600 1400 400 100  4.8695927 4.9175927}%
\special{ar 1600 1400 400 100  5.0615927 5.1095927}%
\special{ar 1600 1400 400 100  5.2535927 5.3015927}%
\special{ar 1600 1400 400 100  5.4455927 5.4935927}%
\special{ar 1600 1400 400 100  5.6375927 5.6855927}%
\special{ar 1600 1400 400 100  5.8295927 5.8775927}%
\special{ar 1600 1400 400 100  6.0215927 6.0695927}%
\special{ar 1600 1400 400 100  6.2135927 6.2615927}%
}}%
%
{\color[named]{Black}{%
\special{pn 4}%
\special{pa 400 400}%
\special{pa 2400 400}%
\special{fp}%
}}%
%
{\color[named]{Black}{%
\special{pn 4}%
\special{pa 780 780}%
\special{pa 2800 780}%
\special{fp}%
}}%
%
{\color[named]{Black}{%
\special{pn 4}%
\special{pa 1200 1400}%
\special{pa 1200 780}%
\special{fp}%
}}%
%
{\color[named]{Black}{%
\special{pn 4}%
\special{pa 1200 800}%
\special{pa 1200 600}%
\special{dt 0.027}%
}}%
%
{\color[named]{Black}{%
\special{pn 4}%
\special{pa 2000 800}%
\special{pa 2000 600}%
\special{dt 0.027}%
}}%
%
{\color[named]{Black}{%
\special{pn 4}%
\special{pa 400 1200}%
\special{pa 800 1600}%
\special{fp}%
}}%
%
{\color[named]{Black}{%
\special{pn 4}%
\special{pa 2400 1200}%
\special{pa 2800 1600}%
\special{fp}%
}}%
%
{\color[named]{Black}{%
\special{pn 4}%
\special{pa 800 1600}%
\special{pa 2800 1600}%
\special{fp}%
}}%
%
{\color[named]{Black}{%
\special{pn 4}%
\special{pa 400 1200}%
\special{pa 1000 1200}%
\special{fp}%
}}%
%
{\color[named]{Black}{%
\special{pn 4}%
\special{pa 2200 1200}%
\special{pa 2400 1200}%
\special{fp}%
}}%
\put(26.0000,-12.0000){\makebox(0,0){$C_{3,1}^{(2)}$}}%
\put(16.0000,-3.0000){\makebox(0,0){$C_{3,1}^{(3)}$}}%
%
{\color[named]{Black}{%
\special{pn 4}%
\special{ar 1600 1000 400 100  6.2831853 6.2831853}%
\special{ar 1600 1000 400 100  0.0000000 3.1415927}%
}}%
%
{\color[named]{Black}{%
\special{pn 4}%
\special{ar 1600 1000 400 100  3.1415927 3.1895927}%
\special{ar 1600 1000 400 100  3.3335927 3.3815927}%
\special{ar 1600 1000 400 100  3.5255927 3.5735927}%
\special{ar 1600 1000 400 100  3.7175927 3.7655927}%
\special{ar 1600 1000 400 100  3.9095927 3.9575927}%
\special{ar 1600 1000 400 100  4.1015927 4.1495927}%
\special{ar 1600 1000 400 100  4.2935927 4.3415927}%
\special{ar 1600 1000 400 100  4.4855927 4.5335927}%
\special{ar 1600 1000 400 100  4.6775927 4.7255927}%
\special{ar 1600 1000 400 100  4.8695927 4.9175927}%
\special{ar 1600 1000 400 100  5.0615927 5.1095927}%
\special{ar 1600 1000 400 100  5.2535927 5.3015927}%
\special{ar 1600 1000 400 100  5.4455927 5.4935927}%
\special{ar 1600 1000 400 100  5.6375927 5.6855927}%
\special{ar 1600 1000 400 100  5.8295927 5.8775927}%
\special{ar 1600 1000 400 100  6.0215927 6.0695927}%
\special{ar 1600 1000 400 100  6.2135927 6.2615927}%
}}%
%
{\color[named]{Black}{%
\special{pn 8}%
\special{ar 1600 1400 600 150  5.4977871 6.2831853}%
\special{ar 1600 1400 600 150  0.0000000 3.9269908}%
}}%
%
{\color[named]{Black}{%
\special{pn 8}%
\special{ar 1600 600 600 150  0.0000000 6.2831853}%
}}%
%
{\color[named]{Black}{%
\special{pn 4}%
\special{pa 1000 1400}%
\special{pa 1000 780}%
\special{fp}%
}}%
%
{\color[named]{Black}{%
\special{pn 8}%
\special{pa 1000 750}%
\special{pa 1000 600}%
\special{dt 0.045}%
}}%
%
{\color[named]{Black}{%
\special{pn 8}%
\special{pa 2200 750}%
\special{pa 2200 600}%
\special{dt 0.045}%
}}%
%
{\color[named]{Black}{%
\special{pn 4}%
\special{pa 2200 1400}%
\special{pa 2200 780}%
\special{fp}%
}}%
\end{picture}%
\caption{The places in which $\bfx_4$ is ($\vec{x}\in C_{3,1}^{(l)}$, $l=1,2,3$)}
\label{fig:conf4}
\end{figure}
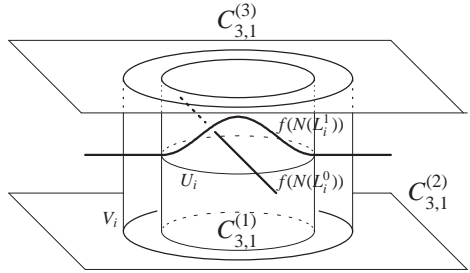
Let $C_{3,1}^{(1)}(f)$ be the subspace of $C_{3,1}(f)$ consisting of $\vec{x}\in C_{3,1}(f)$ with $\bfx_4\in V_i:=V_i'\times[-1,1]$ for some $i\in S$.
Now we compare the integrals of $\omega_Y$ over $C^{(1)}_{3,1}(f)$ and $C^{(1)}_{3,1}(f_S)$.
Since $\vol_{S^{6k-1}}$ is Dirac-type, we only need to consider the set of $\vec{x}$\,'s with $\bfx_1,\bfx_2,\bfx_3\in f^{-1}(V_i)$ for the same $i$ as for $\bfx_4$, since $\omega_Y$ vanishes on other $\vec{x}$\,'s.
The local diffeomorphism $\Phi\colon C_{3,1}^{(1)}(f)\to C_{3,1}^{(1)}(f_S)$,
\[
 \Phi((\bfx_1,\bfx_2,\bfx_3);\bfx_4):=((\bfx_1,\bfx_2,\bfx_3);\iota(\bfx_4)),
\]
reverses the orientation, while $\Phi^*\omega_Y=-\omega_Y$ since $\psi^{\theta}_j\circ\Phi=\iota\circ\psi^{\theta}_j$ and $f_S=\iota\circ f$ on $f^{-1}(V_i)$ ($i\in S$) and $\iota^*\vol_{S^{6k-1}}=-\vol_{S^{6k-1}}$ (because $\vol_{S^{6k-1}}$ is anti-invariant).
Thus the integrations of $\omega_Y$ over $C^{(1)}_{3,1}(f)$ and $C^{(1)}_{3,1}(f_S)$ are equal to each other.

Next, consider the subspace $C_{3,1}^{(2)}(f)$ of $C_{3,1}(f)$ consisting of $\vec{x}$ with $\bfx_4\in\R^{6k-1}\times[-1,1]$ but $\bfx_4\not\in V_i$ for any $i\in S$.
Since $\vol_{S^{6k-1}}$ is Dirac-type and since $\bfx_4\not\in V_i$ for any $i\in S$, $\vec{x}\in C^{(2)}_{3,1}(f)$ can non-trivially contribute to $I(Y)$ only if $\bfx_j$ ($j=1,2,3$) is not contained in any $L_i^1$ ($i\in S$).
Since $f=f_S$ outside $L_i^1$ ($i\in S$), there is no difference between $\omega_Y$'s on $C_{3,1}^{(2)}(f)$ and $C_{3,1}^{(2)}(f_S)$, and the integrals of $\omega_Y$ over $C_{3,1}^{(2)}(f)$ and $C_{3,1}^{(2)}(f_S)$ are also the same.

Finally, consider the subspace $C_{3,1}^{(3)}(f)$ of $C_{3,1}(f)$ consisting of $\vec{x}$ with $\bfx_4\not\in\R^{6k-1}\times[-1,1]$.
If $\delta>0$ is small enough, then under the diffeomorphism $C_{3,1}^{(3)}(f)\to C_{3,1}^{(3)}(f_S)$ given by $\vec{x}\mapsto\vec{x}$, the differences between the vectors $\psi^{\theta}_i(\vec{x})$ ($i=1,2,3$) are small.
This is because $f$ differs from $f_S$ only near $N(L_i^1)$ and the difference is small relative to $\abs{\bfx_4}$.
Thus the difference between the integrals of $\omega_Y$ over $C^{(3)}_{3,1}(f)$ and $C^{(3)}_{3,1}(f_S)$ converges to $0$ in the limit $\delta\to 0$.
\end{proof}

\begin{proof}[Proof of Theorem~\ref{thm:main1}]
Any $f$ with $p\circ f$ generic can be transformed by an ambient isotopy of $\R^{4k-1}$ so that $f$ satisfies the condition in Proposition~\ref{prop:I(X)} without changing the isotopy class of $\bigcup L_i^{\epsilon}$.
Thus we may assume that $f$ satisfies the condition in Proposition~\ref{prop:I(X)}.
Notice
\[
 \calH(f)-\calH(f_S)=\frac{I(X)(f)-I(X)(f_S)}{2}+\frac{I(Y)(f)-I(Y)(f_S)}6+(c(f)-c(f_S)).
\]
The left-hand side does not depend on $\delta$.
In the limit $\delta\to 0$ ($p\circ f$ remains unchanged), the first term of the right-hand side is computed in Proposition~\ref{prop:I(X)} and gives the right-hand side of \eqref{eq:main1}.
The second and the third terms converge to zero by Lemma~\ref{lem:I(Y)}.
\end{proof}

\begin{proof}[Proof of Theorem~\ref{thm:order2}]
Choose three components $A_1,A_2,A_3$ out of $A=A_1\sqcup\dotsb\sqcup A_m$, $m\ge 3$.
Let $W_T(\calH)(f):=\calH(f_T)-\calH(f_{T\cup\{1\}})$ for any $T\subset\{2,3\}$.
Then by \eqref{eq:main2},
\[
 2W_T(\calH)(f)=\sum_{j\ne 1;\,\epsilon,\epsilon'=0,1}(-1)^{\epsilon+\epsilon'}lk(L_1^{\epsilon}(f_T),L_j^{\epsilon'}(f_T)).
\]
Because $L_j^{\epsilon}(f_T)=L_i^{\epsilon+1}(f)$ if $j\in T$ and $L_j^{\epsilon}(f_T)=L_j^{\epsilon}(f)$ otherwise,
\begin{alignat*}{9}
 2W_{\emptyset}(\calH)(f)
 &=\sum_{\epsilon,\epsilon'}(-1)^{\epsilon+\epsilon'}(lk(L_1^{\epsilon},L_2^{\epsilon'})+lk(L_1^{\epsilon},L_3^{\epsilon'}))&	&+\sum_{j\ge 4;\,\epsilon,\epsilon'}(-1)^{\epsilon+\epsilon'}lk(L_1^{\epsilon},L_j^{\epsilon'}),&\\
 2W_{\{2\}}(\calH)(f)
 &=\sum_{\epsilon,\epsilon'}(-1)^{\epsilon+\epsilon'}(-lk(L_1^{\epsilon},L_2^{\epsilon'})+lk(L_1^{\epsilon},L_3^{\epsilon'}))&	&+\sum_{j\ge 4;\,\epsilon,\epsilon'}(-1)^{\epsilon+\epsilon'}lk(L_1^{\epsilon},L_j^{\epsilon'}),&\\
 2W_{\{3\}}(\calH)(f)
 &=\sum_{\epsilon,\epsilon'}(-1)^{\epsilon+\epsilon'}(lk(L_1^{\epsilon},L_2^{\epsilon'})-lk(L_1^{\epsilon},L_3^{\epsilon'}))&	&+\sum_{j\ge 4;\,\epsilon,\epsilon'}(-1)^{\epsilon+\epsilon'}lk(L_1^{\epsilon},L_j^{\epsilon'}),&\\
 2W_{\{2,3\}}(\calH)(f)
 &=\sum_{\epsilon,\epsilon'}(-1)^{\epsilon+\epsilon'}(-lk(L_1^{\epsilon},L_2^{\epsilon'})-lk(L_1^{\epsilon},L_3^{\epsilon'}))&	&+\sum_{j\ge 4;\,\epsilon,\epsilon'}(-1)^{\epsilon+\epsilon'}lk(L_1^{\epsilon},L_j^{\epsilon'}).&
\end{alignat*}
Substituting them into $V_3(\calH)(f)=\sum_{T\subset\{2,3\}}(-1)^{\abs{T}}W_T(\calH)(f)$, we obtain $V_3(\calH)(f)=0$.
\end{proof}
\section{Proof of Theorem~\ref{thm:main2}}\label{s:gen_imm}
\subsection{Well-definedness and invariance of $E$}\label{ss:E_invariant}
Suppose $g\in\Imm{6k-1}{4k-1}$ is generic and liftable, and let $f,f'\in\emb{6k}{4k-1}$ be lifts of $g$.
We can transform $f'$ by an isotopy in the $x_{6k}$-direction (without changing $p\circ f'$) so that $f'=f_S$ for some index set $S$ of the self-intersection of $g$.
Then \eqref{eq:main1} implies that $E(g)$ does not depend on the choice of $f$.

Let $g_t\in\Imm{6k-1}{4k-1}$ ($t\in[0,1]$) be a generic regular homotopy with each $g_t$ liftable.
We show that, for any $t_0\in[0,1]$, $g_t$ can be lifted to an isotopy $f_t\in\emb{6k}{4k-1}$---namely, $g_t=p\circ f_t$---in an open neighborhood of $t_0$.
Let $f_{t_0}$ be a lift of $g_{t_0}$.
Then $f_{t_0}$ can be written as $f_{t_0}=(g_{t_0},h)$ by using some $h\colon\R^{4k-1}\to\R$.
Define $G_t\colon\Conf^{\circ}_2(\R^{4k-1})\to\R^{6k-1}$ by
\[
 G_t(\bfx,\bfy):=g_t(\bfx)-g_t(\bfy).
\]
The first projection $\Conf^{\circ}_2(\R^{4k-1})\to\R^{4k-1}$ restricts to a diffeomorphism $G^{-1}_t(\zero)\cong g_t^{-1}(A_t)$, where $A_t$ is the self-intersection of $g_t$.
Since $g_t$ is a generic regular homotopy, $G^{-1}_t(\zero)$ gives an isotopy of a closed submanifold of $\Conf^{\circ}_2(\R^{4k-1})$.
Because $f_{t_0}\in\emb{6k}{4k-1}$, there exists an open neighborhood $W$ of $G^{-1}_{t_0}(\zero)$ such that $h(\bfx)\ne h(\bfy)$ for any $(\bfx,\bfy)\in W$.
The compactness of $G^{-1}_t(\zero)$ (for any $t$) implies that there exists $\epsilon>0$ such that $G^{-1}_t(\zero)\subset W$ for $\abs{t-t_0}<\epsilon$.
Then $f_t:=(g_t,h)\colon\R^{4k-1}\to\R^{6k}$ is in $\emb{6k}{4k-1}$ for $\abs{t-t_0}<\epsilon$ and is a lift of $g_t$.

Because $\calH(f_t)$ is constant and the linking part of \eqref{eq:E=H-link} is invariant unless the double point set varies, $E(g_t)$ is also constant.
Thus for any $t_0$ there exists $\epsilon>0$ such that $E(g_t)$ is constant on $(t_0-\epsilon,t_0+\epsilon)$, and hence $E(g_t)$ is constant on $[0,1]$.

\begin{remark}
By Proposition~\ref{prop:I(X)} and \eqref{eq:E=H-link}, for a generic liftable $g\in\Imm{6k-1}{4k-1}$,
\begin{equation}\label{eq:E=I(Y)}
 E(g)=\lim_{\delta\to 0}\Bigl(\frac{1}{6}I(Y)(f^{\delta})+c(f^{\delta})+\frac{1}{8}\sum_{i,\epsilon}w(L_i^{\epsilon})\Bigr).
\end{equation}
This gives a geometric interpretation of $I(Y)$ (added by $c$ and the writhes), and is a higher-dimensional analogue to \cite[Definition~5.4]{LinWang96}.
\end{remark}

\subsection{Local models of non-generic self-intersections}
As explained in \cite{Ekholm01,Ekholm01-2}, the set of generic immersions $g\colon\R^{4k-1}\looparrowright\R^{6k-1}$ is an open dense subspace of $\Imm{6k-1}{4k-1}$ and the complement is a stratified hypersurface.
To characterize an invariant of generic immersions, we must study its jumps at non-generic strata.
The codimension $1$ strata (in $\Imm{6k-1}{4k-1}$) consist of immersions with a single generic self-tangency point or a single generic triple point \cite[Lemma~3.4]{Ekholm01-2}.
The local picture of the \emph{versal deformation} \cite[\S5.3]{Ekholm01}, \cite[\S3.2]{Ekholm01-2} of an immersion with a self-tangency or a triple point is given in \cite{Ekholm01, Ekholm01-2}.

\begin{proposition}[{\cite[Lemma~3.5]{Ekholm01-2}}]\label{prop:tangency}
Let $g_0\in\Imm{6k-1}{4k-1}$ be an immersion with a single generic self-tangency point.
Then the versal deformation $g_t$ of $g_0$ is constant far from the self-tangency point, and in some local coordinates near the self-tangency point, $g_t$ is given by
\begin{alignat*}{8}
  g_t(\bfx)&=(x_1,\dotsc,x_{2k},\,& &x_{2k+1},\,& \dotsc,\, &x_{4k-1},\,& &0,                      & &0,         & \dotsc,\,&0),&\\
  g_t(\bfy)&=(y_1,\dotsc,y_{2k},\,& &0,\,       & \dotsc,\, &0,\,       & &Q(y_1,\dots,y_{2k})+t,\,& &y_{2k+1},\,& \dotsc,\,&y_{4k-1}),&
\end{alignat*}
where $Q$ is a non-degenerate quadratic form on $2k$ variables.
\end{proposition}

We say a self-tangency point {\em definite} (resp.\ \emph{indefinite}) if the quadratic form $Q$ is definite (resp.\ indefinite).

\begin{proposition}[{\cite[Lemma~3.6]{Ekholm01-2}}]\label{prop:triplepoint}
Let $g_0\in\Imm{6k-1}{4k-1}$ be an immersion with a single generic triple point.
Then the versal deformation $g_t$ of $g_0$ is constant far from the triple point, and in some local coordinates near the triple point, $g_t$ is given by
\begin{alignat*}{11}
  g_t(\bfx)&=(x_1,\,& \dotsc,\, &x_{2k-1},\,& &x_{2k},\,& &x_{2k+1},\,& \dotsc,\, &x_{4k-1},\,& &0,\,       & &0,\,       & \dotsc,\, &0),&\\
  g_t(\bfy)&=(y_1,\,& \dotsc,\, &y_{2k-1},\,& &0,\,     & &0,\,       & \dotsc,\, &0,\,       & &y_{2k},\,  & &y_{2k+1},\,& \dotsc,\, &y_{4k-1}),&\\
  g_t(\bfz)&=(0,\,  & \dotsc,\, &0,\,       & &z_{2k},\,& &z_1,\,     & \dotsc,\, &z_{2k-1},\,& &z_{2k}-t,\,& &z_{2k+1},\,& \dotsc,\, &z_{4k-1}).&
\end{alignat*}
\end{proposition}

\subsection{The jump of $E$ at a non-generic liftable immersion}
Suppose that $g_0\in\Imm{6k-1}{4k-1}$ is liftable and has a single generic self-tangency point or a single generic triple point, and let $g_t$ be its versal deformation.

We show that $g_t$ is liftable for $\abs{t}$ small.
Let $f_0=(g_0,h)$ be a lift of $g_0$ and $f_t:=(g_t,h)\in\Imm{6k-1}{4k-1}$.
Similarly to the argument in \S\ref{ss:E_invariant}, choose an open neighborhood $W$ of $G_0^{-1}(\zero)$ such that $h(\bfx)\ne h(\bfy)$ for any $(\bfx,\bfy)\in W$.
Then there exists $\epsilon>0$ such that $G_t^{-1}(\zero)\subset W$ for $\abs{t}<\epsilon$;
this follows from the explicit description of the change of the multiple point set of $g_t$ (see below).
Thus $f_t\in\emb{6k}{4k-1}$ for $\abs{t}<\epsilon$ and is a lift of $g_t$.

By the definition \eqref{eq:E=H-link} of the invariant $E$, its jump $E(g_t)-E(g_{-t})$ ($t\ne 0$) is described by the change of linking numbers of $L^{\epsilon}_i$'s because $\calH(f_t)$ remains unchanged.

\subsubsection{Definite self-tangencies}
First, we study the jump of $E$ at a positive definite self-tangency point (the argument needs no change for the negative definite case).
It is clear from Proposition~\ref{prop:tangency} that in some local coordinate near the tangency point, the double point set is given by
\begin{alignat*}{2}
 K^0&=\{(x_1,\dotsc,x_{2k},\zero^{2k-1})\in\R^{4k-1}\mid x_1^2+\dotsb+x_{2k}^2=-t\}\quad& &\text{in the }\bfx\text{-sheet},\\
 K^1&=\{(y_1,\dotsc,y_{2k},\zero^{2k-1})\in\R^{4k-1}\mid y_1^2+\dotsb+y_{2k}^2=-t\}\quad& &\text{in the }\bfy\text{-sheet},
\end{alignat*}
which is empty when $t>0$ and the trivial link when $t<0$.
This link $K^0\sqcup K^1$ is separated from the other links since each $K^{\epsilon}$ is contained in a small open set that intersects no other components of the double point set.
Thus we have the following.

\begin{lemma}\label{lem:jump_definite}
If $g_0$ has a definite self-tangency point, then $E(g_t)=E(g_{-t})$.
\end{lemma}

\subsubsection{Indefinite self-tangencies}
Next, suppose that $g_0$ is liftable and has an indefinite self-tangency point.
Let $0<\lambda<2k$ be the index of $Q$.
By Proposition~\ref{prop:tangency}, in some local coordinates the double point set of $g_t$ near the self-tangency point is given as follows:
\begin{alignat*}{3}
 &\{(x_1,\dotsc,x_{2k},\zero^{2k-1})\in\R^{4k-1}\mid x_1^2+\dotsb+x_{\lambda}^2-x_{\lambda+1}^2-\dotsb-x_{2k}^2=t\}& \ &\text{in the }\bfx\text{-sheet},&\\
 &\{(y_1,\dotsc,y_{2k},\zero^{2k-1})\in\R^{4k-1}\mid y_1^2+\dotsb+y_{\lambda}^2-y_{\lambda+1}^2-\dotsb-y_{2k}^2=t\}& \ &\text{in the }\bfy\text{-sheet}.&
\end{alignat*}
The versal deformation transforms the double point set by a surgery that replaces $S^{\lambda-1}\times D^{2k-\lambda}$ with $D^{\lambda}\times S^{2k-(\lambda+1)}$ (see Figure~\ref{fig:doublept_indef}).
\begin{figure}[htb]
\centering
\input{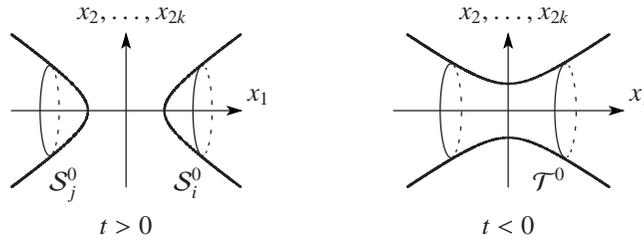}
\caption{The double point set near a self-tangency point of index one}
\label{fig:doublept_indef}
\end{figure}
If $1<\lambda<2k-1$, then this surgery transforms a single component in a small neighborhood of the self-tangency point and changes no linking numbers with other components.
Thus $E(g_t)=E(g_{-t})$.

If $\lambda=1$, then in each sheet the double point set has two components when $t>0$:
in the $\bfx$-sheet
\[
 \calS^0_i:=\left\{x_1= (x_2^2+\dotsb+x_{2k}^2+t)^{1/2}\right\}
 \text{ and }
 \calS^0_j:=\left\{x_1=-(x_2^2+\dotsb+x_{2k}^2+t)^{1/2}\right\},
\]
and in the $\bfy$-sheet
\[
 \calS^1_i:=\left\{y_1= (y_2^2+\dotsb+y_{2k}^2+t)^{1/2}\right\}
 \text{ and }
 \calS^1_j:=\left\{y_1=-(y_2^2+\dotsb+y_{2k}^2+t)^{1/2}\right\}
\]
(Figure~\ref{fig:doublept_indef}, left).
Here we choose $f_t$ so that it maps the $\bfx$-sheet ``below'' the $\bfy$-sheet.
In the $\bfx$-sheet, when $t<0$, $\calS^0_i\sqcup\calS^0_j$ is joined into a single component
\[
 \calT^0:=\{x_1^2-t=x_2^2+\dotsb+x_{2k}^2\}
\]
(Figure~\ref{fig:doublept_indef}, right).
Similarly, in the $\bfy$-sheet $\calS^1_i\sqcup\calS^1_j$ is joined into a single component $\calT^1:=\{y_1^2-t=y_2^2+\dotsb+y_{2k}^2\}$ when $t<0$.

Let $L^{\epsilon}_*$ ($*=i,j$; $\epsilon=0,1$) be the components of the double point set that contains $\calS^{\epsilon}_*$ when $t>0$, and let $K^{\epsilon}$ be the component of the double point set containing $\calT^{\epsilon}$ when $t<0$.

\textbf{Case 1}.
Consider the case $A_i\ne A_j$.
Here the number of the components of the self-intersection decreases by $1$ when $t$ changes from $t>0$ to $-t$.

In this case, two components $L_i^{\epsilon}\sqcup L_j^{\epsilon}$ of the double point set of $g_t$, $t>0$, are joined into a connected double point set $K^{\epsilon}=L_i^{\epsilon}\sharp L_j^{\epsilon}$ of $g_{-t}$ ($\epsilon=0,1$).
Other components $L_m^{\epsilon}$ ($m\ne i,j$; $\epsilon=0,1$) are unchanged.
The jump of the sum of linking numbers is thus
\begin{multline*}
 \sum_{(p,\epsilon)<(q,\epsilon')}(-1)^{\epsilon+\epsilon'}lk(L_p^{\epsilon}(g_t),L_q^{\epsilon'}(g_t))
 -\sum_{(p,\epsilon)<(q,\epsilon')}(-1)^{\epsilon+\epsilon'}lk(L_p^{\epsilon}(g_{-t}),L_q^{\epsilon'}(g_{-t}))\\
 =
  \Bigl(lk(L^0_i,L^0_j)+lk(L^1_i,L^1_j)-\sum_{p,q=i,j}lk(L^0_p,L^1_q)+\sum_{\genfrac{}{}{0pt}{1}{p=i,j;\, m\ne i,j,}{\epsilon,\epsilon'=0,1}}(-1)^{\epsilon+\epsilon'}lk(L^{\epsilon}_p,L_m^{\epsilon'})\Bigr)\\
 -\Bigl(-lk(K^0,K^1)+\sum_{m\ne i,j;\,\epsilon,\epsilon'=0,1}(-1)^{\epsilon+\epsilon'}lk(K^{\epsilon},L_m^{\epsilon'})\Bigr).
\end{multline*}
The connected-sums $K^{\epsilon}=L_i^{\epsilon}\sharp L_j^{\epsilon}$ are taken near the tangency point, and by a small isotopy we may assume that $L_*^{\epsilon}$ ($*=i,j$; $\epsilon=0,1$) satisfy the condition in Lemma~\ref{lem:link_connected_sum}.
Thus by Lemma~\ref{lem:link_connected_sum}
\[
 \sum_{p,q=i,j}lk(L^0_p,L^1_q)=lk(K^0,K^1),\quad
 lk(L^{\epsilon}_i,L_m^{\epsilon'})+lk(L^{\epsilon}_j,L_m^{\epsilon'})=lk(K^{\epsilon},L_m^{\epsilon'}).
\]
By the above three equations, we have
\[
 E(g_t)-E(g_{-t})=\pm\frac{lk(L^0_i,L^0_j)+lk(L^1_i,L^1_j)}{4}.
\]

\textbf{Case 2}.
Consider the case when the versal deformation does not change the number of the components of the self-intersection---namely, $k>1$ and $A_i=A_j$.
Since $L_i^{\epsilon}=L_j^{\epsilon}$ turns into $K^{\epsilon}$ and other components $L_m^{\epsilon}$ ($m\ne i$; $\epsilon=0,1$) are unchanged, $E(g_t)=E(g_{-t})$ follows from the same argument as in the case $1<\lambda<2k-1$.

\textbf{Case 3}.
When $k=1$ and a negative self-tangency occurs at $t=0$ in $A_i=A_j$ (namely, two arcs get tangent to each other with opposite velocity vectors; by Lemma~\ref{lem:orientation} no positive self-tangency occurs), the number of the components of the self-intersection increases by $1$ when $t$ changes from $t>0$ to $-t$.
This case is similar to Case 1.

The case $\lambda={2k}-1$ is similar, and we have the following.

\begin{lemma}\label{lem:jump_indefinite}
Suppose that $g_0$ has an indefinite self-tangency point.
If the index of $Q$ is $1$ or $2k-1$, and if the versal deformation changes the number of the components of self-intersection, then
\[
 E(g_t)-E(g_{-t})=\pm\frac{lk(L_i^0,L_j^0)+lk(L_i^1,L_j^1)}{4},
\]
where $L_i^0\sqcup L_i^1=g_t^{-1}(A_i)$ and $L_j^0\sqcup L_j^1=g_t^{-1}(A_j)$ are the double point set corresponding to $A_i$ and $A_j$, the distinct components of the self-intersection of $g_t$ that are joined into a single component after the versal deformation.
Otherwise, $E(g_t)=E(g_{-t})$.
\end{lemma}

\subsubsection{Triple points}
Suppose $g_0$ is liftable and has a triple point.
By Proposition~\ref{prop:triplepoint}, in some local coordinates, the double point set near the triple point in the $\bfx$-sheet is given by
\begin{alignat*}{10}
 \calS^1_i&:=\{(x_1,& &\dotsc,\,& &x_{2k-1},\,& &0,\,& &0,       & &\dotsc,\,& &0)\}       & &=+\R^{2k-1}\times\{\zero\}^{2k},&\\
 \calS^1_j&:=\{(0,  & &\dotsc,\,& &0,       \,& &t,\,& &x_{2k+1},& &\dotsc,\,& &x_{4k-1})\}& &=\{\zero\}^{2k-1}\times\{t\}\times(+\R^{2k-1})&
\end{alignat*}
as oriented manifolds (see Figure~\ref{fig:triplepoint}).
\begin{figure}
\centering
\unitlength 0.1in
\begin{picture}( 42.0000, 10.5000)(  6.0000,-15.3500)
%
{\color[named]{Black}{%
\special{pn 4}%
\special{pa 800 1000}%
\special{pa 1600 1000}%
\special{fp}%
\special{sh 1}%
\special{pa 1600 1000}%
\special{pa 1534 980}%
\special{pa 1548 1000}%
\special{pa 1534 1020}%
\special{pa 1600 1000}%
\special{fp}%
}}%
%
{\color[named]{Black}{%
\special{pn 4}%
\special{pa 1200 1400}%
\special{pa 1200 600}%
\special{fp}%
\special{sh 1}%
\special{pa 1200 600}%
\special{pa 1180 668}%
\special{pa 1200 654}%
\special{pa 1220 668}%
\special{pa 1200 600}%
\special{fp}%
}}%
%
{\color[named]{Black}{%
\special{pn 13}%
\special{pa 1360 840}%
\special{pa 900 1300}%
\special{fp}%
\special{sh 1}%
\special{pa 900 1300}%
\special{pa 962 1268}%
\special{pa 938 1262}%
\special{pa 934 1240}%
\special{pa 900 1300}%
\special{fp}%
}}%
%
{\color[named]{Black}{%
\special{pn 13}%
\special{pa 1400 1290}%
\special{pa 1400 600}%
\special{fp}%
\special{sh 1}%
\special{pa 1400 600}%
\special{pa 1380 668}%
\special{pa 1400 654}%
\special{pa 1420 668}%
\special{pa 1400 600}%
\special{fp}%
}}%
%
{\color[named]{Black}{%
\special{pn 13}%
\special{pa 1440 760}%
\special{pa 1500 700}%
\special{fp}%
}}%
\put(14.5000,-12.0000){\makebox(0,0)[lt]{$\mathcal{S}_j^1$}}%
\put(8.0000,-10.5000){\makebox(0,0)[lt]{$\mathcal{S}_i^1$}}%
%
{\color[named]{Black}{%
\special{pn 4}%
\special{pa 2400 1000}%
\special{pa 3200 1000}%
\special{fp}%
\special{sh 1}%
\special{pa 3200 1000}%
\special{pa 3134 980}%
\special{pa 3148 1000}%
\special{pa 3134 1020}%
\special{pa 3200 1000}%
\special{fp}%
}}%
%
{\color[named]{Black}{%
\special{pn 4}%
\special{pa 2800 1400}%
\special{pa 2800 600}%
\special{fp}%
\special{sh 1}%
\special{pa 2800 600}%
\special{pa 2780 668}%
\special{pa 2800 654}%
\special{pa 2820 668}%
\special{pa 2800 600}%
\special{fp}%
}}%
%
{\color[named]{Black}{%
\special{pn 13}%
\special{pa 3100 700}%
\special{pa 2500 1300}%
\special{fp}%
\special{sh 1}%
\special{pa 2500 1300}%
\special{pa 2562 1268}%
\special{pa 2538 1262}%
\special{pa 2534 1240}%
\special{pa 2500 1300}%
\special{fp}%
}}%
%
{\color[named]{Black}{%
\special{pn 13}%
\special{pa 2600 1260}%
\special{pa 2600 1400}%
\special{fp}%
\special{sh 1}%
\special{pa 2600 1400}%
\special{pa 2620 1334}%
\special{pa 2600 1348}%
\special{pa 2580 1334}%
\special{pa 2600 1400}%
\special{fp}%
}}%
%
{\color[named]{Black}{%
\special{pn 13}%
\special{pa 2600 1160}%
\special{pa 2600 700}%
\special{fp}%
}}%
%
{\color[named]{Black}{%
\special{pn 4}%
\special{pa 4000 1000}%
\special{pa 4800 1000}%
\special{fp}%
\special{sh 1}%
\special{pa 4800 1000}%
\special{pa 4734 980}%
\special{pa 4748 1000}%
\special{pa 4734 1020}%
\special{pa 4800 1000}%
\special{fp}%
}}%
%
{\color[named]{Black}{%
\special{pn 4}%
\special{pa 4400 1160}%
\special{pa 4400 600}%
\special{fp}%
\special{sh 1}%
\special{pa 4400 600}%
\special{pa 4380 668}%
\special{pa 4400 654}%
\special{pa 4420 668}%
\special{pa 4400 600}%
\special{fp}%
}}%
%
{\color[named]{Black}{%
\special{pn 4}%
\special{pa 4700 700}%
\special{pa 4100 1300}%
\special{fp}%
\special{sh 1}%
\special{pa 4100 1300}%
\special{pa 4162 1268}%
\special{pa 4138 1262}%
\special{pa 4134 1240}%
\special{pa 4100 1300}%
\special{fp}%
}}%
%
{\color[named]{Black}{%
\special{pn 13}%
\special{pa 4800 800}%
\special{pa 4300 1300}%
\special{fp}%
\special{sh 1}%
\special{pa 4300 1300}%
\special{pa 4362 1268}%
\special{pa 4338 1262}%
\special{pa 4334 1240}%
\special{pa 4300 1300}%
\special{fp}%
}}%
%
{\color[named]{Black}{%
\special{pn 4}%
\special{pa 4400 1400}%
\special{pa 4400 1240}%
\special{fp}%
}}%
%
{\color[named]{Black}{%
\special{pn 13}%
\special{pa 4400 1240}%
\special{pa 4400 1360}%
\special{fp}%
\special{sh 1}%
\special{pa 4400 1360}%
\special{pa 4420 1294}%
\special{pa 4400 1308}%
\special{pa 4380 1294}%
\special{pa 4400 1360}%
\special{fp}%
}}%
%
{\color[named]{Black}{%
\special{pn 13}%
\special{pa 4400 800}%
\special{pa 4400 1160}%
\special{fp}%
}}%
\put(30.8000,-7.3000){\makebox(0,0)[lt]{$\mathcal{S}_i^0$}}%
\put(25.5000,-8.5000){\makebox(0,0)[rb]{$\mathcal{S}_p^1$}}%
\put(14.3000,-9.7000){\makebox(0,0)[lb]{$\sb t$}}%
\put(25.8000,-9.8000){\makebox(0,0)[rb]{$\sb{-t}$}}%
\put(46.3000,-10.3000){\makebox(0,0)[lt]{$\sb t$}}%
\put(48.0000,-8.0000){\makebox(0,0)[lb]{$\mathcal{S}_j^0$}}%
\put(43.7000,-8.7000){\makebox(0,0)[rb]{$\mathcal{S}_p^0$}}%
\put(6.0000,-13.3000){\makebox(0,0)[lt]{$\sb{x_1,\dotsc,x_{2k-1}}$}}%
\put(8.0000,-5.0000){\makebox(0,0)[lt]{$\sb{x_{2k+1},\dotsc,x_{4k-1}}$}}%
\put(17.0000,-10.0000){\makebox(0,0){$\sb{x_{2k}}$}}%
\put(12.0000,-16.0000){\makebox(0,0){highest}}%
\put(24.9000,-13.1000){\makebox(0,0)[rt]{$\sb{y_1,\dotsb,y_{2k-1}}$}}%
\put(40.9000,-13.1000){\makebox(0,0)[rt]{$\sb{z_1,\dotsb,z_{2k-1}}$}}%
\put(33.0000,-10.0000){\makebox(0,0){$\sb{y_{2k}}$}}%
\put(28.0000,-5.5000){\makebox(0,0){$\sb{y_{2k+1},\dotsc,y_{4k-1}}$}}%
\put(44.0000,-5.5000){\makebox(0,0){$\sb{z_{2k+1},\dotsc,z_{4k-1}}$}}%
\put(49.0000,-10.0000){\makebox(0,0){$\sb{z_{2k}}$}}%
\put(28.0000,-16.0000){\makebox(0,0){middle}}%
\put(44.0000,-16.0000){\makebox(0,0){lowest}}%
%
{\color[named]{Black}{%
\special{pn 13}%
\special{pa 4400 800}%
\special{pa 4400 670}%
\special{dt 0.045}%
}}%
\end{picture}%
\caption{Double point sets near the triple point}
\label{fig:triplepoint}
\end{figure}
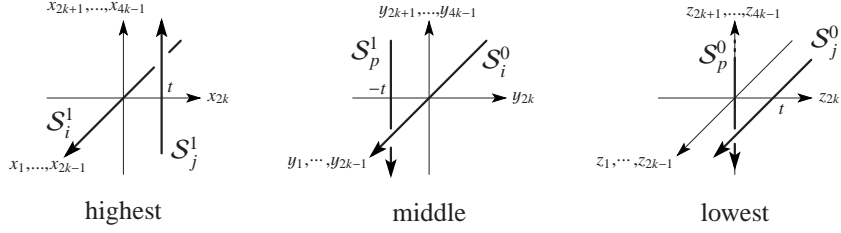
The orientations are direct consequences of Lemma~\ref{lem:orientation}.
Similarly, the double point set in the $\bfy$-sheet is given by
\begin{alignat*}{10}
 \calS^1_p&:=\{(0,  & &\dotsc,\,& &0,\,       & &-t,\,& &y_{2k+1},\,& &\dotsc,\,& &y_{4k-1})\}& &=\{\zero\}^{2k-1}\times\{-t\}\times(-\R^{2k-1}),& \\
 \calS^0_i&:=\{(y_1,& &\dotsc,\,& &y_{2k-1},\,& &0,\, & &0,\,       & &\dotsc,\,& &0)\}& &=+\R^{2k-1}\times\{\zero\}^{2k},&
\end{alignat*}
and in the $\bfz$-sheet by
\begin{alignat*}{10}
 \calS^0_j&:=\{(z_1,& &\dotsc,\,& &z_{2k-1},\,& &t,\,& &0,       & &\dotsc,\,& &0)\}       & &=+\R^{2k-1}\times\{t\}\times\{\zero\}^{2k-1},& \\
 \calS^0_p&:=\{(0,  & &\dotsc,\,& &0,\,       & &0,\,& &z_{2k+1},& &\dotsc,\,& &x_{4k-1})\}& &=\{\zero\}^{2k}\times(-\R^{2k-1}).&
\end{alignat*}
Here, without loss of generality, we assume that the lift $f_t$ of $g_t$ maps the $\bfx$-sheet (in Proposition~\ref{prop:triplepoint}) to the ``highest position,'' the $\bfy$-sheet to the ``middle'' and the $\bfz$-sheet to the ``lowest.''
The following holds by the above descriptions.

\begin{lemma}[{see \cite[Remark~6.2.3]{Ekholm01}}]\label{lem:ori_triple}
The versal deformation changes three crossing $\calS^1_i\sqcup\calS^1_j$, $\calS^0_i\sqcup\calS^1_p$, and $\calS^0_j\sqcup\calS^0_p$, changing their linking numbers or writhes by the common value $\pm 1$ (the signs are same for all the three crossings).
\end{lemma}

Let $L_*^{\epsilon}$ ($*\in\{i,j,p\}$, $\epsilon\in\{0,1\}$) be the component of the double point set containing $\calS_*^{\epsilon}$.

\textbf{Case 1}.
If all the six components $L_*^{\epsilon}$ are different, then by Lemma~\ref{lem:ori_triple} the versal deformation changes $(-1)^{1+1}lk(L^1_i,L^1_j)+(-1)^{0+1}lk(L_i^0,L_p^1)+(-1)^{0+0}lk(L_j^0,L_p^0)$ by $\pm(1-1+1)=\pm 1$.
Other linking numbers do not change.
Thus $E(g_t)-E(g_{-t})=\pm 1/4$.

\textbf{Case 2}.
If $L_i^{\epsilon}=L_j^{\epsilon}\ne L_p^{\epsilon}$ ($\epsilon=0,1$), then the crossing change at $\calS_i^0\sqcup\calS_p^1$ and $\calS_j^0\sqcup\calS_p^0$ changes $(-1)^{0+1}lk(L_i^0,L_p^1)+(-1)^{0+0}lk(L_i^0,L_p^0)$ by $\pm 1\mp 1=0$.
Other linking numbers do not change, and the change of $w(L_i^1)$ (by $\pm 1$) arising from the crossing change at $\calS_i^1\sqcup\calS_j^1$ does not change $E$.
Thus $E(g_t)=E(g_{-t})$.

\textbf{Case 3}.
If $L_p^{\epsilon}=L_i^{\epsilon}\ne L_j^{\epsilon}$ ($\epsilon=0,1$), then the versal deformation changes $lk(L_i^1,L_j^1)-lk(L_i^1,L_i^0)+lk(L_i^0,L_j^0)$ by $\pm(1-1+1)=\pm1$, and $E(g_t)-E(g_{-t})=\pm 1/4$.

\textbf{Case 4}.
The case $L_j^{\epsilon}=L_p^{\epsilon}\ne L_i^{\epsilon}$ is similar to the Case 2 by symmetry, and $E(g_t)=E(g_{-t})$.

\textbf{Case 5}.
If $L_i^{\epsilon}=L_j^{\epsilon}=L_p^{\epsilon}$, then the versal deformation changes $-lk(L_i^0,L_p^1)$ by $\pm 1$, and the changes of $w(L_i^1)$, $w(L_j^0)$ do not affect $E$.
Thus in this case $E(g_t)-E(g_{-t})=\pm1/4$.

Putting them all together, we obtain the following.

\begin{lemma}\label{lem:jump_triple}
Suppose that $g_0$ has a triple point.
Then $E(g_t)=E(g_{-t})$ if $L_i^{\epsilon}=L_j^{\epsilon}\ne L_p^{\epsilon}$ or $L_j^{\epsilon}=L_p^{\epsilon}\ne L_i^{\epsilon}$, $\epsilon=0,1$ (in Figure~\ref{fig:triplepoint}).
Otherwise, $E(g_t)-E(g_{-t})=\pm 1/4$.
\end{lemma}

\subsection{The case $k=1$}\label{ss:k=1}
For $g\in\emb{5}{3}$ the invariant $E$ is essentially the Smale invariant;

\begin{proposition}[{\cite{Ekholm01,Takase07}}]
If $g\in\emb{5}{3}$ (also regarded as $g\in\emb{6}{3}$ by composing $\R^5\hookrightarrow\R^6$), then $\calH(g)=E(g)=-\Omega(g)/12$, where $\Omega\colon\pi_0(\Imm{5}{3})\xrightarrow{\cong}\Z$ is the Smale invariant.
\end{proposition}
\begin{proof}
$\calH=E$ follows from \eqref{eq:E=H-link}, since $g$ has no self-intersection.
As explained in \cite{Ekholm01}, there exists a ``Seifert surface'' for $g$---that is, an embedding $V^4\hookrightarrow\R^5$ that restricts to $g\colon\partial V=\R^3\hookrightarrow\R^5$---and $\Omega(g)=3\sigma(V)/2$, where $\sigma$ denotes the signature.
\cite[Corollary~2.4]{Takase07} states that $\calH(g)=-\sigma(V)/8$ ($e_F=0$ for $g\in\emb{5}{3}$).
\end{proof}

\begin{remark}
If $g\in\emb{5}{3}$, then $E(g)=I(Y)(g)/6+c(g)$ by \eqref{eq:E=I(Y)}.
Thus $\Omega(g)=-2I(Y)(g)-12c(g)$.
\end{remark}

The double point set of generic $g\in\Imm{5}{3}$ is a classical link.
A result of Ogasa \cite{Ogasa02} characterizes which link can be realized as a double point set of an immersion $\R^3\looparrowright\R^5$.

\begin{theorem}[{A special case of \cite[Theorem~1.1]{Ogasa02}}]\label{thm:Ogasa}
For any link $L\subset S^3$, there exist embeddings $g^{\epsilon}\colon S^3\hookrightarrow\R^5$ ($\epsilon=0,1$) such that $(g^{\epsilon})^{-1}(g^0(S^3)\cap g^1(S^3))$ ($\epsilon=0,1$) are isotopic to the given $L$.
Moreover, we can choose $g^{\epsilon}$ to be isotopic to the natural inclusion.
\end{theorem}

Using this, we show the second half of Theorem~\ref{thm:main2} and Corollary~\ref{cor:Ohba}.

\begin{proof}[Proof of Theorem~\ref{thm:main2}, the second half]
We show that an arbitrarily large jump of $E$ at a single indefinite self-tangency can occur.
Any linear combination of Ekholm invariants $J$ and $\mathit{St}$ cannot satisfy this property, because their jumps are bounded \cite{Ekholm01}.

For any two-component link $L=K_1\cup K_2$ in $S^3$, choose $g^{\epsilon}\colon S^3\hookrightarrow\R^5$, $\epsilon=0,1$ as in Theorem~\ref{thm:Ogasa}.
Taking a suitable connected-sum of the standard inclusion $f_0\colon \R^3\hookrightarrow\R^5$ with $g^0$ and $g^1$, we obtain $g:=f_0\sharp g^0\sharp g^1\in\Imm{5}{3}$, which satisfies the following conditions:
\begin{enumerate}[(i)]
\item
	The self-intersection $A=A_1\sqcup A_2$ of $g$ satisfies $g^{-1}(A_i)=K_i^0\cup K_i^1$ with each $K_1^{\epsilon}\cup K_2^{\epsilon}$ included in the $g^{\epsilon}$-part $(g^{\epsilon})^{-1}(A)$ and isotopic to the given link $L$.
\item
	A lift of $g$ exists and can be obtained by lifting $g^1$-part into $\R^5\times\R_+$ and letting $g^0$-part remain inside $\R^5\times\{0\}$.
\end{enumerate}
$K_1^0\cup K_2^0$ is separated from $K_1^1\cup K_2^1$, by (i) above.
Take $q_i\in A_i$ and $p^{\epsilon}_i\in K_i^{\epsilon}$ so that $g^{\epsilon}(p^{\epsilon}_i)=q_i$.
Choose paths $\gamma^{\epsilon}\colon [0,1]\to\R^3\setminus(g^{\epsilon})^{-1}(A)$ from $p_1^{\epsilon}$ to $p_2^{\epsilon}$ in the $g^{\epsilon}$-part.
Then $C:=g(\gamma^0([0,1])\cup\gamma^1([0,1]))$ is a circle in $\R^5$, which can be seen as the image of a trivial knot.
Thus $C$ bounds an embedded $2$-disk $D$ in $\R^5$ whose interior transversely intersects $g$ at finitely many points outside $K_i^{\epsilon}$ (Figure~\ref{fig:indef_tangency}).
\begin{figure}[htb]
\centering
\unitlength 0.1in
\begin{picture}( 16.6500, 14.7000)(  3.3500,-16.7000)
%
{\color[named]{Black}{%
\special{pn 4}%
\special{ar 1200 600 50 200  1.5707963 4.7123890}%
}}%
%
{\color[named]{Black}{%
\special{pn 8}%
\special{ar 800 600 600 270  4.7123890 6.2831853}%
\special{ar 800 600 600 270  0.0000000 1.5707963}%
}}%
%
{\color[named]{Black}{%
\special{pn 8}%
\special{ar 1200 600 50 200  4.7123890 4.8083890}%
\special{ar 1200 600 50 200  5.0963890 5.1923890}%
\special{ar 1200 600 50 200  5.4803890 5.5763890}%
\special{ar 1200 600 50 200  5.8643890 5.9603890}%
\special{ar 1200 600 50 200  6.2483890 6.3443890}%
\special{ar 1200 600 50 200  6.6323890 6.7283890}%
\special{ar 1200 600 50 200  7.0163890 7.1123890}%
\special{ar 1200 600 50 200  7.4003890 7.4963890}%
\special{ar 1200 600 50 200  7.7843890 7.8539816}%
}}%
%
{\color[named]{Black}{%
\special{pn 8}%
\special{ar 1600 600 600 270  1.5707963 4.7123890}%
}}%
%
{\color[named]{Black}{%
\special{pn 8}%
\special{ar 1600 1400 600 270  1.5707963 4.7123890}%
}}%
%
{\color[named]{Black}{%
\special{pn 8}%
\special{ar 800 1400 600 270  4.7123890 6.2831853}%
\special{ar 800 1400 600 270  0.0000000 1.5707963}%
}}%
%
{\color[named]{Black}{%
\special{pn 8}%
\special{ar 1200 1400 50 200  4.7123890 4.8083890}%
\special{ar 1200 1400 50 200  5.0963890 5.1923890}%
\special{ar 1200 1400 50 200  5.4803890 5.5763890}%
\special{ar 1200 1400 50 200  5.8643890 5.9603890}%
\special{ar 1200 1400 50 200  6.2483890 6.3443890}%
\special{ar 1200 1400 50 200  6.6323890 6.7283890}%
\special{ar 1200 1400 50 200  7.0163890 7.1123890}%
\special{ar 1200 1400 50 200  7.4003890 7.4963890}%
\special{ar 1200 1400 50 200  7.7843890 7.8539816}%
}}%
%
{\color[named]{Black}{%
\special{pn 4}%
\special{ar 1200 1400 50 200  1.5707963 4.7123890}%
}}%
%
{\color[named]{Black}{%
\special{pn 13}%
\special{ar 800 1000 400 130  1.5707963 4.7123890}%
}}%
%
{\color[named]{Black}{%
\special{pn 8}%
\special{pa 800 330}%
\special{pa 400 330}%
\special{fp}%
}}%
%
{\color[named]{Black}{%
\special{pn 13}%
\special{ar 1600 1000 400 130  4.7123890 6.2831853}%
\special{ar 1600 1000 400 130  0.0000000 1.5707963}%
}}%
%
{\color[named]{Black}{%
\special{pn 8}%
\special{pa 800 1670}%
\special{pa 400 1670}%
\special{fp}%
}}%
%
{\color[named]{Black}{%
\special{pn 8}%
\special{pa 1600 1670}%
\special{pa 2000 1670}%
\special{fp}%
}}%
%
{\color[named]{Black}{%
\special{pn 8}%
\special{pa 1600 330}%
\special{pa 2000 330}%
\special{fp}%
}}%
\put(12.0000,-3.3000){\makebox(0,0)[lb]{$\sb{A_1}$}}%
\put(12.0000,-16.7000){\makebox(0,0)[lt]{$\sb{A_2}$}}%
\put(20.0000,-4.0000){\makebox(0,0)[lt]{$\sb{g^1(S^3)}$}}%
\put(4.0000,-4.0000){\makebox(0,0)[lt]{$\sb{g^0(S^3)}$}}%
%
{\color[named]{Black}{%
\special{pn 8}%
\special{ar 800 1400 50 270  1.5707963 4.7123890}%
}}%
%
{\color[named]{Black}{%
\special{pn 8}%
\special{ar 800 1400 50 270  4.7123890 4.7873890}%
\special{ar 800 1400 50 270  5.0123890 5.0873890}%
\special{ar 800 1400 50 270  5.3123890 5.3873890}%
\special{ar 800 1400 50 270  5.6123890 5.6873890}%
\special{ar 800 1400 50 270  5.9123890 5.9873890}%
\special{ar 800 1400 50 270  6.2123890 6.2873890}%
\special{ar 800 1400 50 270  6.5123890 6.5873890}%
\special{ar 800 1400 50 270  6.8123890 6.8873890}%
\special{ar 800 1400 50 270  7.1123890 7.1873890}%
\special{ar 800 1400 50 270  7.4123890 7.4873890}%
\special{ar 800 1400 50 270  7.7123890 7.7873890}%
}}%
%
{\color[named]{Black}{%
\special{pn 8}%
\special{ar 1600 1400 50 270  4.7123890 4.7873890}%
\special{ar 1600 1400 50 270  5.0123890 5.0873890}%
\special{ar 1600 1400 50 270  5.3123890 5.3873890}%
\special{ar 1600 1400 50 270  5.6123890 5.6873890}%
\special{ar 1600 1400 50 270  5.9123890 5.9873890}%
\special{ar 1600 1400 50 270  6.2123890 6.2873890}%
\special{ar 1600 1400 50 270  6.5123890 6.5873890}%
\special{ar 1600 1400 50 270  6.8123890 6.8873890}%
\special{ar 1600 1400 50 270  7.1123890 7.1873890}%
\special{ar 1600 1400 50 270  7.4123890 7.4873890}%
\special{ar 1600 1400 50 270  7.7123890 7.7873890}%
}}%
%
{\color[named]{Black}{%
\special{pn 8}%
\special{ar 1600 1400 50 270  1.5707963 4.7123890}%
}}%
%
{\color[named]{Black}{%
\special{pn 13}%
\special{ar 800 600 600 270  0.8374837 1.5707963}%
}}%
%
{\color[named]{Black}{%
\special{pn 13}%
\special{ar 800 1400 600 270  4.7123890 5.4457016}%
}}%
\put(8.0000,-7.7000){\makebox(0,0){$\sb{\gamma^0}$}}%
%
{\color[named]{Black}{%
\special{pn 13}%
\special{ar 1600 600 600 270  1.5707963 2.3041089}%
}}%
%
{\color[named]{Black}{%
\special{pn 13}%
\special{ar 1600 1400 600 270  3.9790764 4.7123890}%
}}%
\put(16.0000,-8.0000){\makebox(0,0)[lb]{$\sb{\gamma^1}$}}%
\put(12.0000,-10.0000){\makebox(0,0){$\sb D$}}%
%
{\color[named]{Black}{%
\special{pn 8}%
\special{ia 1200 1000 100 100  0.0000000 6.2831853}%
}}%
%
{\color[named]{Black}{%
\special{pn 4}%
\special{pa 1750 890}%
\special{pa 1500 1140}%
\special{fp}%
\special{pa 1630 890}%
\special{pa 1360 1160}%
\special{fp}%
\special{pa 1540 860}%
\special{pa 1200 1200}%
\special{fp}%
\special{pa 1180 1100}%
\special{pa 1110 1170}%
\special{fp}%
\special{pa 1110 1050}%
\special{pa 1010 1150}%
\special{fp}%
\special{pa 1230 810}%
\special{pa 900 1140}%
\special{fp}%
\special{pa 1080 840}%
\special{pa 800 1120}%
\special{fp}%
\special{pa 950 850}%
\special{pa 690 1110}%
\special{fp}%
\special{pa 790 890}%
\special{pa 590 1090}%
\special{fp}%
\special{pa 670 890}%
\special{pa 490 1070}%
\special{fp}%
\special{pa 510 930}%
\special{pa 420 1020}%
\special{fp}%
\special{pa 1320 840}%
\special{pa 1250 910}%
\special{fp}%
\special{pa 1430 850}%
\special{pa 1300 980}%
\special{fp}%
\special{pa 1850 910}%
\special{pa 1650 1110}%
\special{fp}%
\special{pa 1930 950}%
\special{pa 1780 1100}%
\special{fp}%
}}%
%
{\color[named]{Black}{%
\special{pn 8}%
\special{pa 400 1000}%
\special{pa 800 1000}%
\special{fp}%
\special{sh 1}%
\special{pa 800 1000}%
\special{pa 734 980}%
\special{pa 748 1000}%
\special{pa 734 1020}%
\special{pa 800 1000}%
\special{fp}%
}}%
\put(12.0000,-8.5000){\makebox(0,0)[lt]{$\sb{q_1}$}}%
\put(12.0000,-11.5000){\makebox(0,0)[lb]{$\sb{q_2}$}}%
\end{picture}%
\caption{The self-intersection of $g$ in the proof of Theorem~\ref{thm:main2}}
\label{fig:indef_tangency}
\end{figure}
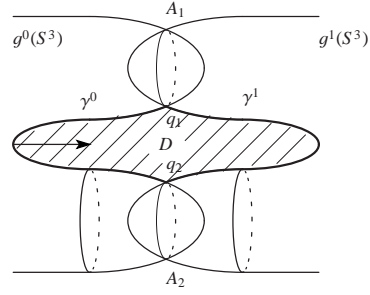
There is a homotopy that transforms $g$ only near $\gamma^0([0,1])$ so that $g(\gamma^0([0,1]))$ gets close to $g(\gamma^1([0,1]))$ along $D$ and eventually $D$ does not intersect $g$ in its interior.
During this homotopy a number of non-generic self-intersections may appear, but $A_1\sqcup A_2$ remains unchanged.
Thus we get $g'\in\Imm{5}{3}$, $g'\simeq g$, for which $\Int D\cap g'=\emptyset$ and we can choose a local coordinate of $\R^5$ around the tubular neighborhood of $D$ so that
\begin{enumerate}[(1)]
\item
	a tubular neighborhood of $g'(\gamma^0([0,1]))$ in $\R^3$ corresponds to
	$\{(x_1,x_2,x_3,0,0)\mid x_1^2+x_3^2<r,\ \abs{x_2}\le 1\}$ ($r>0$ small),
\item
	a tubular neighborhood of $g'(\gamma^1([0,1]))$ in $\R^3$ corresponds to
	$\{(x_1,x_2,0,x_1^2-x_2^2+1,x_5)\mid x_1^2+x_5^2<r,\ \abs{x_2}\le 1\}$, and
\item
	$D$ corresponds to $\{(0,x_2,0,x_4,0)\mid 0\le x_4\le 1-x_2^2\}$
\end{enumerate}
(see Figure~\ref{fig:indef_tangency2}).
\begin{figure}[htb]
\centering
\input{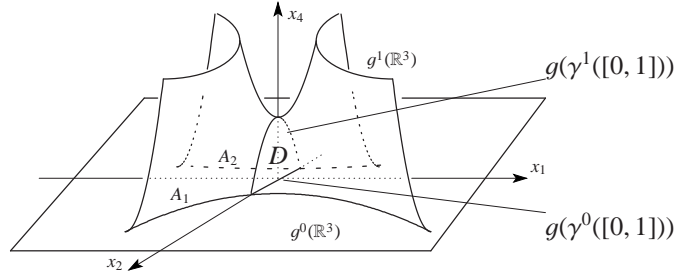}
\caption{A neighborhood of $D$ in $x_1x_2x_4$-plane}
\label{fig:indef_tangency2}
\end{figure}
This coordinate coincides with that in Proposition~\ref{prop:tangency}, and there exists a regular homotopy of $g'$ inside this coordinate in which an indefinite self-tangency between $A_1$ and $A_2$ occurs.
By Lemma~\ref{lem:jump_indefinite}, $E$ jumps by $\pm(lk(K_1^0,K_2^0)+lk(K_1^1,K_2^1))/4=\pm lk(K_1,K_2)/2$ at this self-tangency.
This jump can be arbitrarily large, since $L=K_1\sqcup K_2$ is arbitrary.
\end{proof}

\begin{proof}[Proof of Corollary~\ref{cor:Ohba}]
The immersion $g$ constructed in the proof of Theorem~\ref{thm:main2} can be lifted to $f\in\emb{6}{3}$ by lifting the $g^1$-part into $\R^5\times\R_+$.
Since $g^{\epsilon}$'s are isotopic to the standard embedding, we see that $\calH(f)=0$.
Changing the crossing at $A_1$, by \eqref{eq:main2} we have
\begin{align*}
 0-\calH(f_{\{1\}})&=\frac{1}{2}(lk(K_1^0,K_2^0)-lk(K_1^0,K_2^1)-lk(K_1^1,K_2^0)+lk(K_1^1,K_2^1))\\
 &=lk(K_1,K_2).
\end{align*}
Here we use the fact that $K_i^0$ and $K_j^1$ are separated and that $K_1^{\epsilon}\sqcup K_2^{\epsilon}$ is isotopic to the given link $L=K_1\sqcup K_2$.
This means that an embedding $\R^3\hookrightarrow\R^6$ with arbitrary Haefliger invariant can be obtained by a single crossing change from the trivial embedding.
\end{proof}

\appendix

\section{The Jacobian in the proof of Proposition~\ref{prop:I(X)}}\label{s:Jacobian}
The Jacobian matrix $J(\varphi_X)$ at $\vec{\xi}$ in the proof of Proposition~\ref{prop:I(X)} is given by
\[
\scriptscriptstyle{
 -\left(
 \begin{array}{c|c||c|c||c|c||c|c}
  I_{2k-1} &        & -I_{2k-1} & & & & & \\
 \hline
           & I_{2k} &           & & & & & \\
 \hline
           &        &           & -I_{2k} & & & & \\
 \hline\hline
           &        &           &                   &           &                & -I_{2k-1} & \\
 \hline
           &        &           &                   & I_{2k-1} &                & & \\
 \hline
           &        &           &                   &           & \begin{array}{c|c}I_{2k-1} & \zero \\ \hline {}^t\zero & 1 \\ \hline {}^t\zero & 0\end{array}        &           & \begin{array}{c|c}-I_{2k-1} & \zero \\ \hline {}^t\zero & 0 \\ \hline {}^t\zero & -1\end{array} \\
 \hline\hline
           &        & I_{2k-1}  &                   & -I_{2k-1} & & & \\
 \hline
           &        &           & I_{2k-1}|\,\zero  &           & -I_{2k-1}|\,\zero & &
 \end{array}
 \right)
}
\]
where $I_N$ is the $(N\times N)$-identity matrix and $\zero\in\R^{2k-1}$ is the zero vector.
The rows correspond to the bases of $T_{(e_{6k-1},e_{6k-1},e_{4k-1})}(S^{6k-1}\times S^{6k-1}\times S^{4k-2})$ and the columns correspond to the natural basis of $T_{\vec{\xi}}\Conf_4(\R^{4k-1})\cong T_{\vec{\xi}}\R^{16k-4}$.
Its determinant is $-1$.

\providecommand{\bysame}{\leavevmode\hbox to3em{\hrulefill}\thinspace}
\providecommand{\MR}{\relax\ifhmode\unskip\space\fi MR }
\providecommand{\MRhref}[2]{%
  \href{http://www.ams.org/mathscinet-getitem?mr=#1}{#2}
}
\providecommand{\href}[2]{#2}

\end{document}